\documentclass{amsart}
\usepackage{amssymb, pstricks, wasysym}
\usepackage{hyperref}

\newtheorem{Theorem}{Theorem}[section]
\newtheorem{Proposition}[Theorem]{Proposition}

\newtheorem{Corollary}[Theorem]{Corollary}
\newtheorem{Lemma}[Theorem]{Lemma}
\newtheorem{Fact}[Theorem]{Fact}
\newtheorem{theorem}{Theorem}[section]
\newtheorem{prop}[Theorem]{Proposition}
\newtheorem{lemma}[Theorem]{Lemma}

\newtheorem{fact}[Theorem]{Fact}

\newtheorem*{Claim}{Claim}
\newtheorem*{Claim1}{Claim 1}
\newtheorem*{Claim2}{Claim 2}

\theoremstyle{definition}

\newtheorem{Remark}[Theorem]{Remark}
\newtheorem{remark}[Theorem]{Remark}
\newtheorem{Definition}[Theorem]{Definition}

\newtheorem{defn}[Theorem]{Definition}

\newtheorem{nota}[Theorem]{Notation}
\newtheorem{Remark/defn}[Theorem]{\bf{Remark \& Definition}}

\newsavebox{\indbin}
\savebox{\indbin}{\begin{picture}(0,0)
\newlength{\gnu}
\settowidth{\gnu}{$\smile$} \setlength{\unitlength}{.5\gnu}
\put(-1,-.65){$\smile$} \put(-.25,.1){$|$}
\end{picture}}

\newcommand{\bs}{\boldsymbol}
\newcommand{\be}{\begin{enumerate}}
\newcommand{\bd}{\begin{defn}}
\newcommand{\bt}{\begin{theorem}}
\newcommand{\bl}{\begin{lemma}}
\newcommand{\ee}{\end{enumerate}}
\newcommand{\ed}{\end{defn}}
\newcommand{\et}{\end{theorem}}
\newcommand{\el}{\end{lemma}}

\newcommand{\la}{\langle}
\newcommand{\ra}{\rangle}

\newcommand{\ov}{\overline}

\newcommand{\CL}{{\mathcal L}}

\newcommand{\CM}{{\mathcal M}}

\newcommand{\Aut}{\operatorname{Aut}}
\newcommand{\aut}{\operatorname{Aut}}
\newcommand{\autf}{\operatorname{Autf}}

\newcommand{\p}{\ov{p}}
\newcommand{\wcf}{\operatorname{wcf}}

\def\aute{\operatorname{Aut}_{\bs{e}}}
\def\autfe{\operatorname{Autf}_{\bs{e}}}

\def\autfse{\operatorname{Autf_{\SS}}{(\mathcal{M},\bs{e})}}
\def\Autfse{\operatorname{Autf_{\SS}}{(\mathcal{M},\bs{e})}}
\def\autfKPe{\operatorname{Autf_{KP}}{(\mathcal{M},\bs{e})}}

\def\autfkpe{\operatorname{Autf_{KP}}{(\mathcal{M},\bs{e})}}

\def\LL{\operatorname{L}}
\def\ss{\operatorname{s}}
\def\SS{\operatorname{S}}
\def\KP{\operatorname{KP}}
\def\gal{\operatorname{Gal}}
\def\gall{\operatorname{Gal}_{\LL}}

\newcommand{\id}{\operatorname{id}}
\def\eq{\operatorname{eq}}

\def\dcl{\operatorname{dcl}}

\def\bdd{\operatorname{bdd}}
\def\acl{\operatorname{acl}}
\def\acleq{\operatorname{acl^{eq}}}

\def\tp{\operatorname{tp}}

\def\Ltp{\operatorname{Ltp}}

\def\intr{\operatorname{int}}

\def\Th{\mbox{Th}}

\begin{document}

\title{Automorphism groups over a hyperimaginary}

\author{{Byunghan} \textsc{Kim}}
\address{Department of Mathematics\\
Yonsei University\\ Seoul, Korea}
\email{bkim@yonsei.ac.kr}

\author{Hyoyoon \textsc{Lee}}
\address{Department of Mathematics\\
Yonsei University\\ Seoul, Korea}
\email{alternative@yonsei.ac.kr}

\date{\today}

\subjclass[2020]{Primary 03C60; Secondary 54H11}

\keywords{Lascar group, hyperimaginary, strong types, G-compactness}

\thanks{The authors were supported by  NRF of Korea grants 2018R1D1A1A02085584 and 2021R1A2C1009639}

% \address{Department of Mathematics\\ Yonsei University\\
% 134 Shinchon-dong, Seodaemun-gu\\
% Seoul 120-749, South Korea}

% \email{bkim@yonsei.ac.kr}

 \begin{abstract}
 In this paper we study the Lascar group over a hyperimaginary $\bs e$. We  verify that  various results about the group over a real set still hold when  the set  is replaced by $\bs e$.
  First of all, there
is no written proof in the available literature that the group over $\bs e$ is a topological group. We present an expository style proof of the fact, which even simplifies existing proofs for the  real case. We further extend a result that the orbit equivalence relation under a closed subgroup of the Lascar group is type-definable. On the one hand, we correct errors appeared in \cite[5.1.14-15]{K} and produce a counterexample. 
  On the other, we extend  Newelski's Theorem in \cite{N} that `a  G-compact theory over a set has a uniform bound for the Lascar distances' to the hyperimaginary context.    
 Lastly, we supply a partial positive answer to a question raised in \cite[2.11]{DKKL}, which is even a new result in the real context.
 %This is an expository note on the Lascar group. 
% We also study the Lascar group over hyperimaginaries, and make some new observations on the strong types over those. In %particular we  show that in a simple theory, $\Ltp\equiv\stp$ in {\em real} context implies that for hyperimaginary context.
 \end{abstract}

% \thanks{This is a note submitted for the proceedings of the conference: Recent developments in model theory,
% Ol\'{e}ron, France, June 2011, where the author gave a talk on different subjects in \cite{GKK},\cite{GKK1}.  This work is supported by an NRF grant 2010-0016044.}

\maketitle

The Lascar (automorphism) group of a first-order complete theory and its quotient groups such as the Kim-Pillay group and the Shelah group have been  central themes in  contemporary model theory.  The study on those groups enables us to develop Galois theoretic correspondence between the groups and their orbit-equivalence relations on a monster model such as Lascar types, Kim-Pillay types, and Shelah strong types.  The notions of the  Lascar group and its topology are introduced first by D. Lascar in  \cite{L} using ultraproducts. Later more favorable  equivalent definition is suggested in \cite{K0} and \cite{LP}, which is nowadays considered as a standard approach.  However  even a complete proof using the approach of the fundamental fact that the Lascar group is a topological group is not  so well available. 
For example in \cite{CLPZ}, its proof is left to the readers, while the proof is not at all trivial.  As far as we can see, only in \cite{Z}, a detailed proof is written.  

Aforementioned results are for the Lascar group over $\emptyset$, or more generally over a real set $A$. 
In this paper we study the Lascar group over a hyperimaginary $\bs e$ and verify how  results on the Lascar group over $A$ can be extended to the case when the set is replaced by $\bs e$.   Indeed this attempt was made in \cite{K1} (and rewritten in  \cite[\S 5.1]{K}). However those contain some errors, and moreover a proof of that the Lascar group over $\bs e$ is a topological group is also missing.
 In this paper we supply a proof of the fact in a detailed expository manner. Our proof is more direct and even simplifies that  for the group over $\emptyset$ in \cite{Z}. We correct the mentioned errors in \cite{K},\cite{K1}, as well. In particular we correct the proof of that the orbit equivalence relation under a closed normal subgroup of the Lascar group over $\bs e$ is type-definable over $\bs e$. Moreover we extend  Newelski's Theorem in \cite{N} to the hyperimaginary context. Namely we show  that  if $T$ is G-compact over $\bs e$ then there is $n<\omega$ such that for any hyperimaginaries $\bs b,\bs c$, we have $\bs b\equiv^{\LL}_{\bs e}\bs c$ iff the Lascar distance between $\bs b$ and $\bs c$ is $\leq n$.  
 We also generalize the notions of {\em relativized} Lascar groups introduced in \cite{DKL}, in the context of hyperimaginaries.
 Lastly, we prove a new result even in the real context, which is a partial positive answer to  a question raised in  \cite[2.11]{DKKL}. That is, if $T$ is G-compact over $\bs e$, and $\bs c$ has only finitely many conjugates $\bs c_i$ ($i<m$) over $\bs b\bs e$ such that $\bs b \bs c_i\not\equiv^{\LL}_{\bs e} \bs b \bs c_j$ for $i<j<m$  then, for $p=\tp(\bs b/\bs e)$ and $\p=\tp(\bs b\bs c/\bs e)$, the kernel of the canonical projection $\pi_n:\gall^n(\p)\to \gall^n(p)$ of the relativized Lascar groups is finite. 
 %This implies that the relativized Lascar groups  are isomorphic if $ 

 In Section 1,  we introduce basic terminology for this paper and supply the mentioned detailed proof of that the Lascar group over $\bs e$ is a quasi-compact topological group (Corollary \ref{corlastopgp}).
 
 In Section 2,
 we prove that  the orbit equivalence relation under a closed subgroup of the Lascar group over $\bs e$ is  bounded and type-definable. In addition if the subgroup is normal then the equivalence relation is $\bs e$-invariant (Corollary \ref{maininthesection}).
 
 In Section 3,  using results in Section 2 we study particular quotient groups of the Lascar group. Namely we investigate the Kim-Pillay group and the Shelah group, and corresponding  notions of Kim-Pillay types and Shelah strong types, over a hyperimaginary.  We point out an error occurred in \cite{K1} (rewritten in \cite{K}), by  producing a counterexample. 
 
 In Section 4, using the approach in \cite{P} we prove that any type-definable Lascar  type over $\bs e$ has a finite diameter (Theorem \ref{typedeflstp}), by which we can extend Newelski's result (over $\emptyset$) to the class of  theories being G-compact over $\bs e$ (Corollary \ref{corgcptchar}).  
 
 In Section 5,  we introduce the notions of relativized Lascar groups over $\bs e$, generalizing those for real tuples over $\emptyset$ in \cite{DKL}. Then as mentioned above, we supply an answer covering the hyperimaginary case to a question in \cite{DKKL} (Theorem \ref{thmkernelfinite}).  We also point out that  results in \cite{DKKL} can be extended to our hyperimaginary context (Corollary \ref{k1finite}, Proposition \ref{propabelian}, Theorem \ref{last}).

\section{The Lascar group over a  hyperimaginary}

Throughout this paper, we fix a complete theory $T$ with a language $\CL$ and we work in  a ($\bar \kappa$-saturated) monster model $\CM$ of $T$.
We use standard terminology as in \cite{C} or \cite{K}. For example,  `bounded' or `small' sizes refer to cardinalities  $<  \bar \kappa$.
By {\em real} sets or tuples, we mean  small subsets of $\CM$ or  sequences from $\CM$ of small lengths, respectively. 
Being `type-definable' means being `type-definable over some real set.'
When $A$ is a small set, `$A$-invariant' means `invariant under any automorphism of $\CM$ fixing $A$ pointwise.' 
As is well-known, an $A$-invariant type-definable (solution) set is $A$-type-definable.
Recall that a {\em hyperimaginary}  written as $b_F=b/F$ is  an equivalence class of the real tuple $b$ of an $\emptyset$-type-definable equivalence relation $F$.
Occasionally we may write  hyperimaginaries using boldface letters such as $\bs b,\bs c$.  
%A hyperimaginary is written as $b_F=b/F$ where $b$ is a real tuple and $F$ is  an $\emptyset$-type-definable equivalence relation.
The  {\em type} $\tp_x(b_F/c_L)$ of $b_F$ over another  hyperimaginary $c_L$ is  a partial type $\exists z_1z_2(\tp_{z_1z_2 }( bc) \wedge F(x,z_1) \wedge L(c,z_2))$, whose solution set is the set of  automorphic images of  $b_F$ over  $c_L$. As usual $b_F \equiv_{c_L} d_F$ denotes $\tp(b_F/c_L) = \tp(d_F/c_L)$.
We say $b_F$ and $c_L$ are  {\em interdefinable} or {\em equivalent} if,  any automorphism $f$ fixes the class $b_F$ iff $f$ fixes the class $c_L$.
Real or imaginary tuples  in $\CM^{\eq}$ are examples of hyperimaginaries, but not every hyperimaginary is equivalent with an imaginary tuple. 
Notice that  a sequence of hyperimaginaries $(d_i/F_i\mid i<\lambda)$ is interdefinable with a hyperimaginary $d/F$ where
$d=(d_i\mid i<\lambda)$ and  $F=\bigwedge_{i<\lambda} F_i(x^0_i,x^1_i)$ such that $x^k_i\cap x^l_j=\emptyset$ iff $(k,i)\ne (l,j)$,
where $k,l\in \{0,1\}$.

%Instead of $\tp(b_F/d_L) = \tp(c_F/d_L)$, the notation $b_F \equiv_{d_L} c_F$ is preferred to here, which is equivalent to say that there is an automorphism $f$ of a monster model such that $f(b_F) = c_F$ and $f(d_L) = d_L$ (fixes $d_L$ setwise).

{\bf For the rest of this paper,  we fix  $E$, an $\emptyset$-type-definable equivalence relation, $a$ a real (possibly infinite) tuple, and $\bs{e} = a/E=a_E$ a hyperimaginary.}
For simplicity we may write  a formula in a type, say $ E(x,y)$ as $\varphi(x,y)$ (or $\varphi(a,b)$ for $\models E(a,b)$)  which indeed means $\varphi(x',y')$ for some  suitable finite subtuples $x',y'$ of
$x,y$, respectively. In this section we show  that the Lascar group  over $\bs{e}$ is a quasi-compact   topological group (Corollary \ref{corlastopgp}).
In fact this  is mentioned in \cite[5.1.3]{K} without a proof.   But we cannot find any stated proof  in the literature showing that  the proof
for a real parameter (as in \cite{Z}) can go thru for a hyperimaginary parameter. 
We confirm this,  largely by   following ideas in the real case proof. But the proof here is 
 more direct   (for example Proposition \ref{invcts}, Theorem \ref{thmmulcon}) and simplified. 

%Here we confirm this with a more direct  proof (for example Proposition \ref{invcts}, Theorem \ref{thmmulcon}),  which even supplies considerable simplificati the real case proof.  

 %Finally recall that a compact space is a quasi-compact Hausdorff space.

\begin{Definition}\label{strongaut}$ $
\begin{enumerate}
    \item $\aut_{\bs{e}}(\CM) = \{f \in \aut(\CM) : f(\bs{e}) = \bs{e}\}$.
    \item Let $\bs e'$ be  a hyperimaginary. We denote $\bs e' \in \operatorname{dcl}(\bs e)$ and say $\bs{e}'$ is \emph{definable over} $\bs{e}$ if $f(\bs{e}') = \bs{e}'$ for all $f \in \aut_{\bs e}(\CM)$
    \item Likewise denote $\bs e'\in \bdd(\bs e)$ and say $\bs{e}'$ is \emph{bounded over} $\bs{e}$ if 
    $$\{ f(\bs e')\mid f \in \aut_{\bs e}(\CM)\}$$
    is bounded.  
    \item $\operatorname{Autf}_{\bs{e}}(\CM) $ is the subgroup of $\aut_{\bs{e}}(\CM)$ generated by 
    %the sets of the form
    %$\{f \in \aut_{M}(\CM) \mid \bs{e} \in \dcl(M)\}.$
      $$\{f\in \aute(\CM)\mid f \in \aut_{M}(\CM) \mbox{ for some model } M \mbox{ such that } \bs{e} \in \dcl(M)\}.$$
\end{enumerate}
\end{Definition}

The following remarks will be freely used.

\begin{Remark}\label{bddequiv}$ $
\begin{enumerate}
    \item A hyperimaginary $b_F$ is called {\em countable} if $|b|$ is countable.
    Recall that any hyperimaginary is interdefinable with a sequence of  countable hyperimaginaries \cite[Lemma 4.1.3]{K}.
    Thus the \emph{definable closure} of $\bs{e}$, $\dcl(\bs{e})$ can be regarded as the set of all countable hyperimaginaries which are definable over $\bs{e}$. Then  $c_L$ is definable over $\bs e$  iff $c_L$ is  equivalent  with a sequence of countable hyperimaginaries from  $\dcl(\bs{e})$.
    Moreover the (bounded) set  $\dcl(\bs e)$  is interdefinable with a single hyperimaginary, as indicated at the end of the first paragraph of this section.
    
    %\item In the sense of (1), say $\dcl(\bs{e}) = \{b_i/F_i \mid i \in I\}$ where $b_i/F_i$'s are countable hyperimaginaries in $\dcl(\bs{e})$.
    %Then it is interdefinable with a single hyperimaginary $b_F \in \dcl(\bs{e})$ where $b = (b_i \mid i \in I)$ and $F(x,y) := \bigwedge_{i \in I}F_i(x_i,y_i)$.
    \item Likewise, the \emph{bounded closure} of $\bs{e}$, $\bdd(\bs{e})$ can be regarded as
    the set of all countable hyperimaginaries which are bounded over $\bs e$. Again then  $\bdd(\bs e)$ is equivalent with 
    a single hyperimaginary bounded over  $\bs{e}$.

\item It easily follows that $\autf_{\bs{e}}(\CM)$ is a normal subgroup of $\aut_{\bs{e}}(\CM)$.
\end{enumerate}
\end{Remark}

\begin{Definition}
$\gall(\CM, \bs{e}) = \aut_{\bs{e}}(\CM) / \autfe(\CM)$, and 
$\pi: \aute(\CM)\to \gall(\CM, \bs{e})$ is the canonical projection. 
For $f\in\aute(\CM)$, $\ov f$ denotes 
$\pi(f)= f\cdot \autfe(\CM)$.
\end{Definition}

%First Goal : The group structure of $\operatorname{Gal}_{\operatorname{L}}(\CM, \bs{e})$ does not depend on the choice of $\CM$, up to isomorphism.

\begin{remark}\label{preli} $ $
\be\item \cite[Lemma 1.6]{LP} The following are equivalent.
\be\item $\bs{e}\in\dcl(M)$.
\item $ E(x,a)$ is finitely satisfiable in  $M$.
\item $\bs{e}\in \bdd(M)$.
\ee

\item
Recall that  hyperimaginaries $b_F,c_F$ have the same {\em Lascar (strong) type} over $\bs{e}$,  denoted by $b_F \equiv^{\LL}_{\bs{e}} c_F$ or $\Ltp(b_F/\bs{e}) = \Ltp(c_F/\bs{e})$, if there is $f \in \operatorname{Autf}_{\bs{e}}(\CM)$ such that $f(b_F) = c_F$.
 
 As is well-known the following are equivalent (so $\equiv^{\LL}_{\bs e}$ does not depend on the choice of a monster model).
\be\item $b_F\equiv^{\LL}_{\bs{e}}c_F$.

\item There are $1 \leq n < \omega$, models $M_i $ $(i<n)$ and reals $b_i$ $(i\leq  n)$ such that $\bs e\in \dcl(M_i)$, $b=b_0,\  c=b_n$,  and $b_i/F\equiv_{M_i} b_{i+1}/F$.

\item  There are $1 \leq n < \omega$, reals $b=b_0,\dots, b_n=c$ such that for each $i<n$,  $b_i/F$ and $b_{i+1}/F$ begin an $\bs e$-indiscernible sequence.

%\item For any $\bs e$-invariant bounded equivalence relation $L$ coarser than $F$, $L(b,c)$ holds.  
\ee 
%\medskip
Moreover $x_F\equiv^{\LL}_{\bs{e}}y_F$ is the finest  $\bs e$-invariant bounded equivalence relation coarser than $F$.

\item
For small models $M,N (\prec \CM)$ such that $\bs{e}\in \dcl(M), \dcl(N)$,  if  $f(M)\equiv_N g(M)$ where $f,g \in \aute(\CM)$,
 then $\ov{f}=\ov{g}$ (see for example \cite[Remark 5.1.2(1)]{K}).
\ee
\end{remark}

We now point out  the following. 

\begin{Theorem}\label{gallindmon}
$\operatorname{Gal}_{\operatorname{L}}(\CM, \bs{e})$,  up to isomorphism, does not depend on the choice of $\CM$.
\end{Theorem}

The same proof for the real parameter case  (e.g. \cite[Remark 5.1.2]{K}) would work, so  we omit it.

\begin{Definition}
We call $\gall(\CM,\bs e)$ the {\em Lascar group} of $T$ over $\bs{e}$. Due to Theorem \ref{gallindmon},  we omit $\CM$ and  denote it by
$
\gall(T,\bs{e}).
$
%where $\CM$ is any monster model of $T$ such that $\bs e\in \dcl(\CM)$.
\end{Definition}

\begin{Corollary}\label{corgalsize}
$|\operatorname{Gal}_{\operatorname{L}}(T,\bs{e})| \leq 2^{|T| + |a|}$.
\end{Corollary}
\begin{proof}
By Remark \ref{preli}(3), any $\overline{f}$ is determined by any $M,N \models T$ such that $\bs{e} \in \operatorname{dcl}(M), \operatorname{dcl}(N)$.
Then $|\{\operatorname{tp}(f(M)/N) : f \in \operatorname{Aut}_{\bs{e}}(\CM)\}| \leq 2^{|M| + |N| + |T|}$, and letting $M,N$ contain $a$ and have cardinality $|T| + |a|$, we see that $|\operatorname{Gal}_{\operatorname{L}}(T,\bs{e})| \leq 2^{|T| + |a|}$.
\end{proof}

\begin{Definition}
Let $M, N \prec \CM$ such that  $\bs e\in\dcl(M),\dcl(N)$.
Define $S_{M,\bs e}(N)=S_M(N): = \{\operatorname{tp}(f(M)/N) : f \in \operatorname{Aut}_{\bs{e}}(\CM)\}$.
% i.e. conjugates of $\operatorname{tp}(M/N)$ over $\bs{e}$.
\end{Definition}

We give  the relative Stone topology on $S_M(N)$:
As usual, let $S_x(N) = \{p(x) : |x| = |M|$ and $p$ is a complete type over $N\}$ be the Stone space
(that is,  the basic open sets are  of the form $[\varphi] = \{p \in S_x(N) : \varphi \in p\}$ for $\varphi \in \CL(N)$).
Then $S_M(N)$ is a subspace of $S_x(N)$. We remark as follows that 
$S_M(N)$ is a compact (i.e. quasi-compact and Hausdorff) space.

\begin{Remark} Let $r(x,a)$ be a partial type over $a$ that describes $\tp_x(M/\bs e)$.
Since $\bs e\in \dcl(N)$, we have $r_0(x,N):=\tp(a/N)\models E(x,a)$.
Hence $\tp_x(M/\bs e)$ can also be described by
$$r'(x,N)=\exists y (r(x,y)\wedge r_0(y,N)).$$
%where $r_0(y', N)=\tp (a/N)$. 
Thus $S_M(N)=\{p \in S_x(N) : r'(x,N) \subseteq p\}$ is a closed (so compact) subspace of the compact space $S_x(N)$.
\end{Remark}

Due to Remark \ref{preli}(3), there are well-defined maps $\mu,\nu$
\[\begin{matrix}
    \operatorname{Aut}_{\bs{e}}(\CM) \overset{\mu}{\rightarrow} S_M(N) \overset{\nu}{\rightarrow} \operatorname{Gal}_{\operatorname{L}}(T, \bs{e}) \\
    f \mapsto \operatorname{tp}(f(M)/N) \mapsto \overline{f}
\end{matrix}\]
 such  that $\pi=\nu\circ \mu$. 
We give $\gall(T,\bs{e})$ the quotient topology with respect to $\nu$.
Notice that  $S_M(N) \overset{\nu \text{ : cts}}{\longrightarrow} \operatorname{Gal}_{\operatorname{L}}(T,\bs{e})$ implies $\operatorname{Gal}_{\operatorname{L}}(T,\bs{e})$ is quasi-compact (but not necessarily Hausdorff). 

\medskip

We point out another independence. 
%that the topology of $\operatorname{Gal}_{\operatorname{L}}(T,\bs{e})$ is free from
 %the choice of small models $M,N$.

\begin{Proposition}\label{indepmn}
The topology on $\operatorname{Gal}_{\operatorname{L}}(T,\bs{e})$ does not depend on the choice of $M$ and $N$, i.e. any pair of  models $M'$ and $N'$ such that $\bs{e} \in \operatorname{dcl}(M'), \operatorname{dcl}(N')$ induces the same topology.
\end{Proposition}

Again the proof is the same as the real parameter  case, so we omit.  

\medskip

Now  we  begin to show that $\gall(T,\bs e)$ is a topological group with the equipped topology.
Our proof of this  is more direct than that in  \cite{Z}.
Especially, we do not appeal to the argument extending the language. 
Recall that a group $G$ equipped with a topology is called a topological group if the group operations $G \times G \rightarrow G : (x,y) \mapsto xy$, $G \rightarrow G : x \mapsto x^{-1}$ are continuous.

\medskip

%\begin{Corollary}\label{corlasgam}
%Having the same Lascar type over $\bs{e}$ is the transitive closure of $\Gamma$.
%\end{Corollary}

Thanks to Proposition \ref{indepmn}, from now on, we  \textbf{fix a small model $M$ containing $a$} and work with the relative compact  Stone space $S_{M,\bs e}(M)=S_M(M)$.
%Letting $p,q \in S_M(M)$, $f'(M) = M' \models p$ and $f''(M) = M'' \models q$ where $f',f'' \in \operatorname{Aut}_{\bs{e}}(\CM)$
%(Note that $\bs{e} \in \operatorname{dcl}(M'), \operatorname{dcl}(M'')$ by Lemma \ref{efixed}),
As said,  we have the following maps:
\[\operatorname{Aut}_{\bs{e}}(\CM) \overset{\mu}{\longrightarrow} S_M(M) \overset{\nu}{\longrightarrow} \operatorname{Gal}_{\operatorname{L}}(T,\bs{e})\]
\[f \mapsto p=\tp(f(M)/M)  \mapsto \nu(p) =\ov f\]
%\[f'' \mapsto q = \operatorname{tp}(M''/M) = \operatorname{tp}(f''(M)/M) \mapsto \nu(q) = \overline{f''}\]

\medskip

\begin{Proposition}\label{invcts}
The inverse operation on $\gall(T,\bs e)$ is continuous.
\end{Proposition}
\begin{proof}
Let $C$ be a closed subset of $\gall(T,\bs e)$. We want to show $B=\{\ov{f}\in \gall(T,\bs e)\mid \ov{f^{-1}}\in C\}$ is closed. 
Now assume $\nu^{-1}(C)(\subseteq S_M(M))$ is type-defined by say $\Phi(y)$ over $M$. Consider a type over $M$
$$\Psi(x)=\exists y ((xM\equiv_{\bs e} My)\wedge \Phi(y)).$$
Let $B_0=\{p(x)\in S_M(M)\mid \Psi(x)\subseteq p(x)\}$, which  is closed.

It suffices to see that $B_0=\nu^{-1}(B)$: Let $q=\tp(M_0/M)\in B_0$. Then there is $M_1$ such that 
$M_0M\equiv_{\bs e} MM_1$ and $\tp(M_1/M)\in \nu^{-1}(C)$. Now there is $f\in\aute(\CM)$ such that $f(MM_1)=M_0M$. Thus $f^{-1}(M)=M_1$ and 
$\ov{f^{-1}}\in C$. Therefore $\ov{f} \in B$, and $q\in \nu^{-1}(B)$ since $q=\tp(f(M)/M)$.

Conversely, let $q = \tp(f(M)/M) \in \nu^{-1}(B)$.
Then $\tp(f^{-1}(M)/M) \in \nu^{-1}(C)$, thus $f^{-1}(M) \models \Phi(y)$.
Also, clearly $f(M)M \equiv_{\bs e} Mf^{-1}(M)$, hence $q = \tp(f(M)/M) \in B_0$.
\end{proof}

Showing the product map of  $\gall(T,\bs e)$ is continuous is more complicated and we need preparatory work. Recall that 
$x_0x_1\dots x_n\equiv_{\bs e} x_{i_0}\dots x_{i_n}$ means
$ \exists y ((x_0\dots x_na\equiv  x_{i_0}\dots x_{i_n}y) \wedge E(y,a)). $

\begin{Definition}\label{Gamdef}
Let $\Gamma(x_0,x_1)$ be a partial type over $a$ saying there is an $\bs{e}$-indiscernible sequence beginning  with $x_0, x_1$, that is 
$$\exists  x_2x_3\dots (x_0x_1\dots x_n\equiv_{\bs e} x_{i_0}\dots x_{i_n}\mid  i_0<\dots < i_n\in \omega, \ 1\leq n ). $$ 
%(Recall that 
%$x_0x_1\dots x_n\equiv_{\bs e} x_{i_0}\dots x_{i_n}$ means
%$$ \exists y ((x_0\dots x_na\equiv  x_{i_0}\dots x_{i_n}y) \wedge E(y,a)). )$$
\end{Definition}

We will freely use the following.

\begin{Remark}\label{remgaminv} $ $
\be\item
    We can assume that  each formula in $\Gamma$ is reflexive and symmetric since $\Gamma$ itself is so, and $\Gamma$ is   closed under finite conjuction of formulas. 
    \item
    Due to Remark \ref{preli}(2), $x\equiv^{\LL}_{\bs e} y$ is  the transitive closure of $\Gamma(x,y)$.
  % \item Assume $\delta\in \Gamma$, $\bs e\in \dcl(A)$ and $p(x)\in S(A)$. Then there is $n\in\omega$ such that there are  at most length $n$ antichain  $c_0,\dots , c_n$  of $%  %\delta$ in  $p$ (i.e.
    %$c_i\models p$ and $\neg\delta(c_i,c_j) $ for $i<j\leq n$):  Otherwise by Ramsey, there is an $A$-indiscernible sequence $I=\la c_i\mid i<\omega\ra$ in $p$
    %such that $\neg \delta (c_0,c_1)$ holds. Then since $I$ is $\bs e$-indiscernible as well $\Gamma (a_0, a_1)$ must hold, a contradiction.
    
     \item Assume $\delta\in \Gamma$. Then there is $n\in\omega$ such that there exists  at most length $n$ antichain  $c_0,\dots , c_n$  of $\delta$  (i.e.
     $\neg\delta(c_i,c_j) $ for $i<j\leq n$):  Otherwise by Ramsey, there is an $a$-indiscernible sequence $I=\la c_i\mid i<\omega\ra$ 
    such that $\neg \delta (c_0,c_1)$ holds. Then since $I$ is $\bs e$-indiscernible as well $\Gamma (c_0, c_1)$ must hold, a contradiction.
    \ee 
\end{Remark}

\begin{Remark}\label{remnulas} 
For $p_i\in S_M(M)$  $(i=0,1$), we write  $p_0 \approx p_1$ if $\nu(p_0) = \nu(p_1)$.
Then 
$p_0 \approx p_1$ iff   $M_0 \equiv^{\operatorname{L}}_{\bs{e}} M_1$ for any(some) $M_i\models p_i$.
\end{Remark}

\begin{proof}  Assume  $M_i\models p_i$. Then there are  $f_i\in\aute(\CM)$ such that   $f_i(M)=M_i$.
Now assume $p_0\approx p_1$,  so  $\ov{f_0}=\ov{f_1}$.
 Hence there is $g \in \operatorname{Autf}_{\bs{e}}(\CM)$ such that $f_0= gf_1$, so $M_0 =f_0(M)= g(f_1(M))=g(M_1)$.
Therefore  $M_0 \equiv^{\operatorname{L}}_{\bs{e}} M_1$.

Conversely assume
 $M_0 \equiv^{\operatorname{L}}_{\bs{e}} M_1$. Then there is $g \in \operatorname{Autf}_{\bs{e}}(\CM)$ such that $M_0 = g(M_1)$,
  so $\operatorname{tp}(f_0(M)/M) = \operatorname{tp}(g(f_1(M))/M)$.
Thus   $\overline{f_0} = \ov{gf_1}=\overline{f_1}$ by Remark \ref{preli}(3),  and $\nu(p_0) = \nu(p_1)$.
\end{proof}

\begin{Remark/defn}\label{drel} $ $
\be
\item  Let 
$$S_{M^2}(M):=\{\tp(f(M)g(M)/M) \mid f,g \in \aute(\CM)\}.$$
Then the map $r: S_{M^2}(M)\to S_M(M)\times S_M(M)$ 
sending $\tp(M_0M_1/M)$ to 
$(\tp(M_0/M),\tp(M_1/M))$ is a closed map since it is a continuous map between compact spaces.
\item 
We define a binary relation $D$ on $S_M(M)$ by, $D(p,q)$ holds iff there are $M' \models p$, $M'' \models q$ such that $\models \Gamma(M', M'')$, i.e.
$D(x,y)$ is type-defined by 
$$\tp_x(M/\bs{e})  \wedge \tp_y(M/\bs{e})  \wedge  \Gamma(x,y).$$
Then by Remark \ref{remgaminv}(2) and \ref{remnulas}, $p \approx q$ is the transitive closure of $D$.
Now by (1), $D$ is a closed subset of $S_M(M) \times S_M(M)$.
\item For $p\in S_M(M)$, we let $D[p] = \{q \in S_M(M) : D(p,q)\}$, i.e.
$D[p](x)$ is type-defined by $\exists y(p(y)\wedge \Gamma(x,y)).$
 %the set of types in $S_M(M)$ having $D$-relation with $p$.
It is closed so compact in $S_M(M)$. For $q\in D[p]$, $\nu(p)=\nu(q)$.

\item For $\delta(x',y')=\delta(x',y',a)\in \Gamma$, let 
$$D_\delta[p]=\{q\in S_M(M)\mid \exists y (p(y) \wedge \delta(x',y'))\subseteq q(x)\}.$$
Note that $p\in D[p]=\bigcap_{\delta\in\Gamma} D_\delta[p]$. Hence by compactness for any open $O\subseteq S_M(M)$,
$D[p]\subseteq O$ iff for some $\delta\in\Gamma$, $D_\delta[p]\subseteq O$.
\ee
\end{Remark/defn}

\begin{Lemma} \label{intlem}
Let $p\in S_M(M)$.
For any $\delta \in \Gamma$, $p \in \intr (D_{\delta}[p])$.
\end{Lemma}
\begin{proof}
As in Remark \ref{remgaminv}(3), there is a maximal antichain
$\la M_0, \dots, M_n \ra $  of $\delta$  in $p$, so that for any $M' \models p$, there is $i$ such that $\models \delta(M_i,M')$.
Then $p(y) \models \bigvee^n_{i=0} \delta(M_i,y)$, so there is $\psi \in p$ such that $\psi(y) \models \bigvee^n_{i=0} \delta(M_i,y)$.
Note that $p \in [\psi]$, which is open.
We claim $[\psi] \subseteq D_{\delta}[p]$: For any $q \in [\psi]$, since $\psi \models \bigvee^n_{i=0}\delta(M_i,y)$, if $M' \models q$, then there is $j$ satisfying $\delta(M_j,M')$.
Thus  $q\in D_\delta[p]$, so   $p \in \intr(D_{\delta}[p])$.
\end{proof}

The following lemma and corollary  will play a critical role as in \cite{Z}.  We modify/simplify the proof.

\begin{Lemma} \label{prekeylem}
 Assume $W(\subseteq S_M(M))$ is closed under $\approx$, i.e. $W=\nu^{-1}(\nu(W))$.
Let $O=\intr(W)$, and let $U=\{p\in S_M(M)\mid D[p]\subseteq O \}(\subseteq O)$.
\be\item
For $p\in O$, the following are equivalent.
\be\item
$\nu^{-1}(\nu(p))\subseteq O$, i.e. $q\in O$ for any $q\approx p$.
\item
$D[p]\subseteq O$
\item  $D_\delta[p]\subseteq O$ for some $\delta\in \Gamma$.
\ee
\item
$U=\{p\in S_M(M)\mid \nu^{-1}(\nu(p)) \subseteq O\}$ so $U=\nu^{-1}(\nu(U))$, and $U$ is open.
\ee
\end{Lemma}
\begin{proof}
(1) By Remark \ref{drel}(2)(3)(4), we have  (a)$\Rightarrow$(b)$\Leftrightarrow$(c).  Assume (b) and (c). We will show (a). 
Let $p=\tp(H/M) $ and $D_{\delta}[p] \subseteq O$ for some $\delta(x,y,a) \in \Gamma$, and let $q=\tp(K/M)\approx p$.
We need to show $q \in O$. 

We first claim that there is $\delta'\in \Gamma$ such that $D_{\delta'} [q]\subseteq W$:
Notice that $H \equiv^{\LL}_{\bs{e}} K$ by Remark \ref{remnulas} and  there is $f \in \operatorname{Autf}_{\bs{e}}(\CM)$ such that $f(H) = K$.
Since $\Gamma$ is $\bs e$-invariant, we have $\Gamma\equiv f(\Gamma)\models \delta(x,y,f(a))$. Hence there is $\delta'(x,y,a)\in \Gamma$ such that
$\delta'(x,y,a)\models \delta(x,y,f(a))$. Let $q' \in D_{\delta'}[q]$. It suffices to see $q'\in W$.
Since there is $K' \models q'$ such that $\models \delta' (K,K',a)$, we have $\delta(K,K',f(a))$.
Let
$H' = f^{-1}(K')$ and $p' = \operatorname{tp}(H'/M)$.
Then $\delta(K,K', f(a))$ implies $\models \delta(H,H',a)$ and since $H \models p$, we have $p' \in D_{\delta}[p](\subseteq W)$.
Note that $K' \models q'$ and $f^{-1}(K') \models p'$ where $f \in \operatorname{Autf}_{\bs{e}}(\CM)$, so $p' \approx q'$. Since $W$ is closed under $\approx$,
we have  $q' \in W$ as wanted.

Thus by Lemma \ref{intlem}, $q\in \intr(D_{\delta'}[q])\subseteq O$.
\medskip

(2)  That $U=\{p\in S_M(M)\mid \nu^{-1}(\nu(p)) \subseteq O\}$ follows from (1). Thus obviously  $U=\nu^{-1}(\nu(U))$. Now
say $O = \bigcup_{i<\lambda}[\psi_i]$, so that $S_M(M) \setminus O = \{r \in S_M(M) : \{\neg \psi_i : i < \lambda\} \subseteq r\}$.
    Notice then that  $q\notin U(\subseteq O)$ iff $D[q]\not\subseteq O$
     iff $q(x) \wedge \Gamma(x,y) \wedge  \{\neg\psi_i(y)\}_{i<\lambda}$ is consistent.
    Hence  $S_M(M) \setminus U$ is closed.
\end{proof}

\begin{Corollary}\label{keylem}
Let $(p \in)D[p] \subseteq O \subseteq S_M(M)$ where $O$ is open.
Then $\nu(p) \in \operatorname{int}(\nu(O))$.
\end{Corollary}
\begin{proof}
Put $W=\nu^{-1}(\nu(O))$, so that $W$ is closed under $\approx$ and $O\subseteq \intr(W)$.
As in Lemma \ref{prekeylem}, let $(p\in) U=\{q\in S_M(M)\mid D[q]\subseteq \intr(W)\}(\subseteq \intr(W))$.
Then by Lemma \ref{prekeylem}(2), $U$ is open and 
$\nu^{-1}(\nu(U))=U$. Thus $\nu(U)$ is open and $\nu(p)\in \nu(U)\subseteq \nu(W)=\nu(O)$.
\end{proof}

%The following together with
%Proposition \ref{invcts} completes the proof of that   $\operatorname{Gal}%_{\operatorname{L}}(T,\bs{e})$ is a quasi-compact topological group.

Now we are ready to show our goal.

\begin{Theorem}\label{thmmulcon}
The multiplication on $\operatorname{Gal}_{\operatorname{L}}(T,\bs{e})$ is continuous. 
\end{Theorem}

\begin{proof}
Denote $(G,.)=\gall(T,\bs{e})$. Let $C$ be a closed subset of $G$, say $\nu^{-1}(C)$ is type-defined over $M$ by $\Phi(z)$.
We want to show
 $$B:=.^{-1}(C)=\{(\ov f, \ov g)\mid \ov{fg}\in C\}$$ is closed in $G^2$.
Now let 
$$B_0=\{(p,q)\in S_M(M)\times S_M(M)\mid (\nu(p),\nu(q))\in B\}.$$
\begin{Claim}$ $\be
\item$B_0$ is $\approx$-closed, i.e. if $(p_0,p_1)\in B_0$ and $q_i\approx p_i$ $(i=0,1)$ then $(q_0,q_1)\in B_0$. 
\item
$B_0$ is closed in $S_M(M)^2$.
\ee
\end{Claim}
\begin{proof}[Proof of Claim]
(1) is clear by the definition of $B_0$.

(2) Consider  a type over $M$
$$\Psi(x,y)=\exists z(xz\equiv_{\bs e} My \wedge M\equiv_{\bs e} y \wedge\Phi(z) ).$$
Let $B_1\subseteq S_{M^2}(M)$ be type-defined by $\Psi$, i.e. 
$B_1=\{p(x,y)\in S_{M^2}(M)\mid \Psi(x,y)\subseteq p(x,y)\}$, which is closed in $S_{M^2}(M)$.

Due to Remark \ref{drel}(1), it suffices to show that $B_0=r(B_1)$: 
Let
$\tp(M_0M_1/M)\in B_1$. Hence  $M \equiv_{\bs e} M_1$ and there is $M_2$ such that 
$M_0M_2 \equiv_{\bs e} MM_1$ and $\tp(M_2/M)\in \nu^{-1}(C)$. 
Thus there are $f,g\in \aute(\CM)$ such that
$M_1=g(M)$ and $M_0M_2=f(MM_1)$. Then $fg(M)=f(M_1)=M_2$. Therefore   $\ov{fg}\in C$, and $(\tp(M_0/M),\tp(M_1/M))=(\tp(f(M)/M),\tp(g(M)/M))\in B_0$ as wanted.

For  the converse,
let $(\tp(f(M)/M),\tp(g(M)/M)) \in B_0$, so that $(\overline{f},\overline{g}) \in B$ by definition of $B_0$. Hence $\overline{fg} \in C$.
Now $f(M)fg(M) \equiv_{\bs{e}} Mg(M) \wedge M \equiv_{\bs{e}} g(M) \wedge \Phi(fg(M))$, and $\tp(f(M)g(M)/M)\in B_1$.
\end{proof}

We are ready to show $B$ is closed in $G^2$. Let $(\ov{f_0},\ov{f_1})\in G^2\setminus B$, and let
$W=\{(p,q)\in S_M(M)^2 \mid \nu(p)=\ov{f_0},\ \nu(q)=\ov{f_1}\}$. Hence $B_0\cap W=\emptyset$ and 
again $W$ is $\approx$-closed.  Fix $(p_0,p_1)\in W$. Now we have  $D[p_0]\times D[p_1]\subseteq W$  (Remark \ref{drel}(3)). 
Since both $B_0$ and $D[p_0]\times D[p_1]$ are closed so compact in $ S_M(M)^2$, by  purely topological arguments, there are open neighborhoods $U_i(\subseteq S_M(M))$ of $D[p_i]$ ($i=0,1$)  such that $B_0\cap (U_0 \times U_1)  = \emptyset$. Hence by Lemma \ref{keylem},
$\nu(p_i)=\ov{f_i}\in O_i:=\intr(\nu(U_i))\subseteq G$, so $(\ov{f_0},\ov{f_1})\in O_0\times O_1$.
Now since $B_0\cap (U_0 \times U_1)  = \emptyset$, by Claim (1)
 $B_0\cap (\nu^{-1}(\nu(U_0))\times  \nu^{-1}(\nu(U_1)))=\emptyset$ as well. Hence we must have
$B \cap (O_0\times O_1)=\emptyset$.  
Therefore $B$ is closed in $G^2$.
\end{proof}

\begin{Corollary}\label{corlastopgp}
The Lascar group over a hyperimaginary $\bs{e}$, $\gall(T,\bs{e})$ is a quasi-compact  topological group.
\end{Corollary}

\begin{proof}
By Proposition \ref{invcts} and Theorem \ref{thmmulcon}.
\end{proof}

\section{Type-definability of orbit equivalence relations}

Given a subgroup $H\leq \aut_{\bs e}(\CM)$, we write $x\equiv^H_{\bs e} y$ for $|x|=|y|=\alpha$ to denote the orbit equivalence relation 
 on $\CM^\alpha$ under $H$, that is for $b,c\in \CM^\alpha$, $b\equiv^H_{\bs e} c$ iff there is $h\in H$ such that $c=h(b)$. More generally if $F$ is an $\emptyset$-type-definable
 equivalence relation on $\CM^\alpha$, then $x_F\equiv^H_{\bs e}y_F$ is the orbit equivalence relation on $\CM^\alpha/F$ under $H$, which is coarser than $F$.

In this section we aim to show that 
given closed $H'\leq \gall(T,\bs e)$ and $H=\pi^{-1}(H')\leq \aut_{\bs e}(\CM)$,  
$x_F\equiv^H_{\bs e} y_F$ is a type-definable bounded equivalence relation; and in addition if 
$H'\unlhd \gall(T,\bs e)$ then $x_F\equiv^H_{\bs e} y_F$ is $\bs e$-invariant  (Corollary \ref{maininthesection}). 
This result is claimed to be proved in \cite{K1} (restated in \cite[Lemma 5.1.6(1)]{K}) but the proof contains an error: There for each complete hyperimaginary type $p$ over a hyperimaginary, $\Psi_p(x,y)$  type-defines  a bounded equivalence relation on the solution set of $p$, but the description after this there need not work to extend $\Psi_p(x,y)$ to the whole monster model keeping it type-definable and bounded.

 When  $\bs e=\emptyset$ and 
$H'=\ov{\{\id\}} \unlhd \gall(T)$, a correct proof   using an ultraproduct argument   is given in \cite[Lemma 4.18]{LP}. 
Here we supply a direct proof (not appealing to the ultraproduct method) of the general case result while we  still utilize ideas in \cite{LP}. 

\medskip

We recall the following folklore  first.

\begin{Fact} \label{fklorefct}
Let $F$ be a type-definable equivalence relation on $\CM^\alpha$ which is $\bs e$-invariant.  Then  there is an  $\emptyset$-type-definable equivalence relation
$F'$ such that for any $b\in \CM^\alpha$, 
$b_F$ and $ba/F'$ are interdefinable over $\bs e$.
\end{Fact}
\begin{proof} Since $F$ is $\bs e(=a/E)$-invariant, $F$ is type-definable over $a$ by say $F(x,y;a)$. Then for $p(x)=\tp(a)$, we put 
$$F'(xz,yw)\equiv (F(x,y;z)\wedge E(z,w)\wedge p(z)\wedge p(w))\vee xz=yw.$$
It is not hard to check that $F'$ satisfies the statement. 
\end{proof}

The equivalence of (1) and (3) of the following proposition is proved   in \cite[Lemma 1.9]{LP}  when $\bs e=\emptyset$. 
%for the real parameter case.
The argument is essentially the same, but  modifications are made to handle the hyperimaginary parameter.
%We supply a more direct proof for the hyperimaginary context.  

\begin{Proposition}\label{PropHbbFTFAE}
Let $b$ be any small tuple in $\CM$, and let $H \leq \aut_{\bs e}(\CM)$.
The following are equivalent.
\begin{enumerate}
    \item $H = \operatorname{Aut}_{b_F\bs{e}}(\CM)$ for some $\bs{e}$-invariant type-definable equivalence relation $F$.
   
    \item  $H = \operatorname{Aut}_{(ba/L) \bs{e}}(\CM)$ for some  $\emptyset$-type-definable equivalence relation $L$
    such that $ba/L\in\dcl(b\bs e)$.
    
     \item $\operatorname{Aut}_{b\bs{e}}(\CM) \leq H $ and the orbit of $b$ under $H$ is type-definable.
\end{enumerate}
\end{Proposition}
\begin{proof}
Let $X$ be the orbit of $b$ under $H$.

(1) $\Rightarrow$ (2): This follows from Fact \ref{fklorefct}.

\medskip

(2) $\Rightarrow$ (3): Since  $ba/L\in\dcl(b\bs e)$, we have 
 $\operatorname{Aut}_{b\bs{e}}(\CM) \leq H $. Moreover $X$ is type-defined by
 $\exists y (L(xy, ba)\wedge xy\equiv_{\bs e} ba).$

\medskip

(3) $\Rightarrow$ (1): Since $X$ is invariant under $H$, it is also invariant under $\operatorname{Aut}_{b\bs{e}}(\CM)$, so type-definable by some $b\bs{e}$-invariant partial type $\Phi(x,b,a)$  \ (*) where $\Phi(x,y,z) \subseteq \CL$.
We define 
$$F^*(x,y) \equiv \exists z(\Phi(x,y,z) \wedge \tp_x(b/\bs{e}) \wedge \tp_{yz}(ba/\bs{e})).$$

\begin{Claim}
$F^*(x,y)$ is an equivalence relation on the solution set of $\tp(b/\bs{e})$.
\end{Claim}

\begin{proof}[Proof of Claim]
Let  $d \equiv_{\bs e} b$ be given, hence $d= f(b)$ for some $f \in \aute(\CM)$.
Then $\vDash \Phi(b,b,a)$ implies $\vDash \Phi(d,d,f(a))$ and so $F^*(d,d)$ holds.

To show symmetry, assume $F^*(c,d)$ holds for some  $c,d\models \tp(b/\bs e)$.  Thus there is
$a'$ such that $da'\equiv_{\bs e} ba$ and $\Phi(c,d,a')$ holds. Hence there also is 
 $c'$ such that 
$cd\equiv_{\bs e} c'b$,  and for some $g\in \aut_{\bs e}(\CM)$  we have 
$g(cda')=c'ba''$ and $\Phi(c',b,a'')$ holds.   Note that $ba\equiv_{\bs e} b a''$ and as said in (*) above,  $\Phi(x,b,a'')$ also type-defines $X$. Hence $c'\in X$ and  there is $f\in H$ such that  $c'=f(b)$. Note that  
$f^{-1}(b)\in X$ as well and $\Phi(f^{-1}(b),b,a'')$ holds. Hence by applying $g^{-1}.f\in \aut_{\bs e}(\CM)$ we have 
$\Phi(d, c, a''')$ where $a'''= g^{-1}.f(a'')$ and $ca'''\equiv_{\bs e} ba$.  Therefore we have $F^*(d,c)$.

For transitivity, assume $F^*(u,v)$ and $F^*(v,w)$. Hence there are $a_0,a_1$ such that $va_0\equiv_{\bs e} ba\equiv_{\bs e}wa_1$ and 
$\Phi(u,v,a_0)$, $\Phi(v,w,a_1)$ hold. Now there is $u'v'a_2$ such that $u'v'a_2ba\equiv_{\bs e} uva_0wa_1$.
Hence $\Phi(u',v', a_2)$ and $\Phi(v',b,a)$ hold, and $v'\in X$. Thus there is $h\in H$ such that $h(b)=v'$, and we have
$\Phi(u'', b, a_3)$ where $u''=h^{-1}(u')$, $a_3= h^{-1}(a_2)$. Notice that $ba_3\equiv_{\bs e} v'a_2\equiv_{\bs e} ba$, and hence again by (*), we have
$u''\in X$. Therefore there is $k\in H$ such that $k(b)=u''=h^{-1}(u')$, and $u'=h.k(b)\in X$. Hence  $\Phi(u',b,a)$, and  $\Phi(u,w,a_1)$ hold.
Thus $F^*(u,w)$ follows because $ba\equiv_{\bs e} wa_1$.
\end{proof}

Define $F(x,y) \equiv F^*(x,y) \vee x=y$, so that $F$ is a type-definable  equivalence relation on $\CM^{|x|}$.
Notice that $F$ is $\bs{e}$-invariant. It remains to show 
$H = \operatorname{Aut}_{b_F\bs{e}}(\CM).$

If $h \in H$, then   $h(b)\in X$ and  $\Phi(h(b),b,a)$ holds. Thus $F(h(b),b)$.
Conversely, assume $h \in \operatorname{Aut}_{b_F\bs{e}}(\CM)$.
Then $F(h(b),b)$, so $F^*(h(b),b)$ and $\Phi(h(b), b,a')$ hold where $ba'\equiv_{\bs e}ba$.
Hence again by (*), we have $h(b)\in X$ and there is $g\in H$ such that $h(b)=g(b)$.
Then $g^{-1}.h\in \aut_{b\bs e}(\CM)\leq H$ since $g^{-1}.h(b)=b$. Therefore 
$h\in H$.
\end{proof}

%The following proposition is proved in exactly the same way as in \cite{LP}, Lemma 4.12.

\begin{Proposition}\label{PropHcld}
Let $H\leq \aut_{\bs e}(\CM)$ and let $H'=\pi(H)\leq \gall(T,\bs e)$.
Then   $H'$ is closed in $\gall(T,\bs{e})$
and $H=\pi^{-1}(H')$, if and only  if $H=\Aut_{\bs{e}'\bs{e}}(\CM)$ for some hyperimaginary $\bs{e}' \in \bdd({\bs{e}})$. 
\end{Proposition}

\begin{proof}
$(\Rightarrow)$:
Recall that (before Proposition \ref{invcts}) we have fixed $M \models T$ such that $a \in M$.
Note that
 $\aut_M(\CM)=\aut_{M\bs e}(\CM)\leq \autfe(\CM) \leq H$, and since $H'$ is closed, $\{h(M) : h \in H\}$ is type-definable over $M$.
Thus by Proposition \ref{PropHbbFTFAE}, $H = \operatorname{Aut}_{M_F\bs{e}}(\CM)$ for some $\emptyset$-type-definable equivalence relation $F$. 

It remains to show  that 
$M_F \in \bdd(\bs{e})$:
Notice that due to Corollary \ref{corgalsize}, $[\aute(\CM) : H]=\lambda$ is small. Thus
there is $\{f_i\in \aute(\CM)\mid i<\lambda\}$
such that $\aute(\CM) = \bigsqcup_{i<\lambda} f_i \cdot H$.
But for all $g, h \in \aute(\CM)$, if $g \cdot H= h \cdot H$, then $h^{-1}g\in H$ and hence
$g(M_F) = h(M_F)$.

\medskip

$(\Leftarrow)$: Without loss of generality, say
$\bs{e}' = a'/E'$  where $a' \in M$. Note that for $q(x')=\tp(a'/\bs e)$, since 
$\bs{e}' \in \bdd({\bs{e}})$,
$$ F(x',y'):= (q(x')\wedge q(y')\wedge E'(x',y'))\vee (\neg q(x')\wedge\neg q(y'))$$
is an $\bs e$-invariant bounded equivalence relation on $\CM^{|a'|}$.  Then 
due to the last statement of  Remark \ref{preli}(2) with real $x, y$ of arity $|a'|$, we have $\autf_{\bs e}(\CM)\leq H$. 
Thus $\pi^{-1}(H') = H$.
%since
%$\autf_{\bs e}(\CM)\leq H$
%due to the last statement of  Remark \ref{preli}(2) with real $x, y$ of arity $|a'|$. 
Moreover, $H = \{f \in \aute(\CM) : f(a') \models E'(x',a')\}$, so $\nu^{-1}(H') = \mu\pi^{-1}(H') = \{p(x) \in S_M(M) : E'(x',a') \subseteq p(x)\}$ where $x' \subseteq x$, which is closed.
\end{proof}

\begin{Corollary}\label{maininthesection}
Let $H'\leq \gall(T,\bs{e})$ be closed, and let  $F$ be an $\emptyset$-type-definable equivalence relation.    Then for $H = \pi^{-1}(H')$,  $x_F\equiv^H_{\bs{e}}y_F$ is 
equivalent  to $x_F\equiv_{\bs e' \bs e}y_F$ for some hyperimaginary $\bs e'\in\bdd (\bs e)$, and hence 
$x_F\equiv^H_{\bs{e}}y_F$ is 
an $\bs{e}'\bs{e}$-invariant type-definable bounded equivalence relation.
%for some hyperimaginary $\bs e'\in\bdd (\bs e)$.
Especially, if $H' \unlhd \gall(T,\bs{e})$, then $x_F\equiv^H_{\bs{e}}y_F$ is $\bs{e}$-invariant. 
\end{Corollary}

\begin{proof}
Due to Proposition \ref{PropHcld}, $H=\Aut_{\bs{e}'\bs{e}}(\CM)$ for some hyperimaginary $\bs{e}' =d_F\in \bdd({\bs{e}})$. Hence the $H$-orbit equivalence 
relation $x_F\equiv^H_{\bs e}y_F$ is simply $x_F\equiv_{\bs e'\bs e}y_F$ which is   type-definable, bounded (since $x_F\equiv_{da} y_F$ implies $x_F\equiv_{\bs {e}'\bs e} y_F$), and $\bs{e}'\bs{e}$-invariant.
In addition if $H\unlhd \aut_{\bs e}(\CM)$ then it  easily follows  that  $x_F \equiv^H_{\bs{e}}y_F$ is $\bs{e}$-invariant.
\end{proof}

\section{Strong types over a hyperimaginary}

In this section we reconfirm results  in \cite[5.1.6-18]{K} (excerpted from \cite{K1}),  while correcting errors in  \cite[5.1.14,15]{K} and supplying a counterexample to
\cite[5.1.15]{K}. 

We begin to equip a topology on $\aute(\mathcal{M})$ by pointwise convergence; basic open sets are of the form $O_{u,v} = \{f \in \aute(\mathcal{M}) : f(u) = v\}$ where $u,v$ are some {\em finite real} tuples.
The proof of \cite[Lemma 29]{Z} can go through in the hyperimaginary context;

\begin{Proposition}\label{proppiconti}
The projection map $\pi : \aute(\CM) \rightarrow \gall(T,\bs{e})$ is continuous.
\end{Proposition}

\begin{proof}
Let $U$ be an open subset of $\gall(T,\bs{e})$ and $\ov f = f\cdot \autfe(\CM) \in U$.
Since $\nu$ is a quotient map, $\nu^{-1}(U)$ is open, so there is a basic open set $V_{\varphi(x')} = \{p \in S_M(M) : \varphi(x') \in p\} \subseteq \nu^{-1}(U)$ such that $\tp(f(M)/M) \in V_{\varphi(x')}$. Let $u\in M$ be the finite tuple corresponding to 
 the finite tuple $x'$ of variables. 
Then $\mu^{-1}(V_{\varphi(x')}) = \{g \in \aute(\mathcal{M}) : g(u) \models \varphi(x')\}$ contains $f$ and $f(u) \models \varphi(x')$.
Notice that the basic open set $O_{u,f(u)} = \{h \in \aute(\mathcal{M}) : h(u) = f(u)\}$ contains $f$ and is contained in $\{g \in \aute(\mathcal{M}) : g(u) \models \varphi(x')\} = \mu^{-1}(V_{\varphi(x')})$, implying that $\pi^{-1}(U)$ is open.
\end{proof}

\begin{Corollary}\label{corfinarity}
Let $H'$ be a closed subgroup of $\gall(T, \bs{e})$, $H = \pi^{-1}(H') \leq \aute(\mathcal{M})$ and let $b, c$ be any small real tuples.
Then
\begin{align*}
H & = \{f \in \aute(\mathcal{M}): f\text{ fixes all the }\equiv^H_{\bs e}\text{-classes of any hyperimaginaries}\}\\
& = \{f \in \aute(\mathcal{M}): f\text{ fixes all the }\equiv^H_{\bs e}\text{-classes of any finite real tuples}\}.
\end{align*}
Moreover, $b \equiv^H_{\bs e} c$ iff $b' \equiv^H_{\bs e} c'$ where $b',c'$ are corresponding finite subtuples of $b,c$ respectively.
\end{Corollary}

\begin{proof}
%The first equality in the first statement is trivial.
For the equalities, clearly it suffices to show 
 the second equality. Notice that by Proposition \ref{proppiconti}, $H$ is closed in $\aute(\mathcal{M})$.
Thus $f \in H$ iff for every basic open set $O \subseteq \aute(\mathcal{M})$ containing $f$, $O \cap H$ is nonempty.
Recall that every $O$ is of the form $O_{u,v}$ for finite $u,v$;  and $f \in O_{u,v}$ iff $f(u) = v$. Thus $f \in H$ iff for each finite real $u$, there is $g \in H$ such that $g(u) = f(u)$.
But it is equivalent to say that $u \equiv^H_{\bs e} f(u)$ for every finite real tuple $u$.
%which proves the first statement.

The last statement follows by compactness, since  Corollary \ref{maininthesection} says  $x \equiv^H_{\bs e} y$ is equivalent to  $x\equiv_{\bs{e}'\bs{e}}y$ for some  $\bs{e}' \in \bdd(\bs{e})$. 
\end{proof}

\medskip

We now define   and characterize the KP(Kim-Pillay)-type over a  hyperimaginary.

\begin{Definition} $ $
%Let $b_F$ and $c_F$ be any hyperimaginaries.
\begin{enumerate}
    \item Denote 
    $$\autfkpe = \pi^{-1}(\overline{\{\operatorname{id}\}})$$ 
 where $\overline{\{\operatorname{id}\}}$ is the (topological) closure of the identity in $\gall(T,\bs{e})$.

      \item We let $\equiv^{\KP}_{\bs{e}}$ be the  orbit equivalence relation under $\autfkpe$. Namely, for hyperimaginaries
      $b_F, c_F$, 
    \[b_F \equiv^{\KP}_{\bs{e}} c_F \text{ iff }b_F \equiv^{H}_{\bs e} c_F\]
    where  $(\autf_{\bs e}(\CM)\leq )H=\autfkpe(\leq \aut_{\bs e}(\CM))$.
    We call the equivalence class $b_F / \equiv^{\KP}$ the \emph{KP-type} of $b_F$ over $\bs{e}$. Obviously $b_F \equiv^{\LL}_{\bs{e}} c_F$ implies  $b_F \equiv^{\KP}_{\bs{e}} c_F$.
    
    %$\overline{\{\operatorname{id}\}}$ is the (topological) closure of $\{\operatorname{id} / \autfe(T)\}$ in $\gall(T,\bs{e})$.
\end{enumerate}
\end{Definition}

%We recall some basic facts.
%Lemma \ref{lemhypcble} is well-known(\cite[Lemma 4.1.3]{K}) so it will not be proved here and we remark that some arguments of 
%Lemma \ref{lembddinterdef} is from \cite[Chapter 15]{C}.
%Recall that a hyperimaginary $b_F$ is called countable if $|b| \leq \omega$.

\begin{remark}\label{remidclnormal}
Since $\gall(T,\bs{e})$ is a topological group,  $\overline{\{\operatorname{id}\}}$ is a normal closed subgroup of $\gall(T,\bs{e})$ (see \cite{HM}). 
We denote
$$\gal_{\KP}(T,\bs e):=\gall(T,\bs e)/\ov{\{\id\}}=\aut_{\bs e}(\CM)/\autfkpe.$$
Hence $\gal_{\KP}(T,\bs e)$ is a compact topological group.
\end{remark}

Now we characterize $\equiv^{\KP}_{\bs{e}}$ and find equivalent conditions.
Some arguments are from \cite[Section 5.1]{K}.

\begin{prop}\label{propcharofKP}
Let $F$ be an $\emptyset$-type-definable equivalence relation.  
\begin{enumerate}
    \item $x_F \equiv^{\KP}_{\bs{e}} y_F$ is an $\bs{e}$-invariant type-definable bounded equivalence relation which is coarser than $F$.
    %so that there is an $\bs{e}$-invariant partial type $\Psi(x,y)$ such that $b_F \equiv^{\KP}_{\bs{e}} c_F$ iff $bc \models \Psi(x,y)$.
    \item $x_F \equiv^{\KP}_{\bs{e}} y_F$ is the finest among the $\bs{e}$-invariant type-definable bounded equivalence relations which are coarser than $F(x,y)$.
\end{enumerate}
\end{prop}

\begin{proof}
%\begin{enumerate}
    (1) This follows directly from   Corollary \ref{maininthesection} and Remark \ref{remidclnormal}.
    
    \medskip 
    
    (2) Let $L$ be any $\bs{e}$-invariant type-definable bounded equivalence relation coarser than $F$ and assume that $b_F \equiv^{\KP}_{\bs{e}} c_F$.
    It suffices to show that $L(b,c)$.
    Now  by Fact \ref{fklorefct}, there is a hyperimaginary $ba/L'$ such that $ba/L'$ and $b_L$ are interdefinable over $\bs{e}$.
    Thus $ba/L' \in \bdd(\bs{e})$ and by Proposition \ref{PropHcld}, $\pi(\aut_{(ba/L')\bs{e}}(\mathcal{M}))$ is closed in $\gall(T, \bs{e})$ and $\pi^{-1}(\pi(\aut_{(ba/L')\bs{e}}(\mathcal{M})) )= \aut_{(ba/L')\bs{e}}(\mathcal{M})$.
    Hence 
    %the closure of $\{\operatorname{id}/\autfe(\mathcal{M})\}$ is a subset of $\pi(\aut_{(ba/L')\bs{e}}(\mathcal{M}))$ and 
    $\autfkpe \leq  \aut_{(ba/L')\bs{e}}(\mathcal{M})$.
    Now $b_F \equiv^{\KP}_{\bs{e}} c_F$ implies there is $f \in \aut_{(ba/L')\bs{e}}(\mathcal{M})$ such that $f(b_F) = c_F$ and this $f$ fixes $b_L$ by the interdefinability over
    $\bs{e}$, thus we have $L(b, f(b))$ and $F(f(b),c)$, resulting $L(b,c)$ (since $L$ is coarser than $F$).
%\end{enumerate}
\end{proof}

\begin{prop}\label{propKPTFAE}
Let    $b_F, c_F$ be  hyperimaginaries.
%such that $q(x)=\tp(b)$ and 
%$q'(x)=\tp(b_F)=\exists x'(F(x,x')\wedge q(x'))$.
The following are equivalent.
\begin{enumerate}
    \item $b_F \equiv^{\KP}_{\bs{e}} c_F$.
    \item $b_F \equiv_{\bdd(\bs{e})} c_F$.
   % \item $b_F \equiv_{d_L} c_F\text{ for any }d_L \in \bdd(\bs{e}).$
    \item For any $\bs{e}$-invariant type-definable equivalence relation $L$ which is coarser than $F$, if $b_L$ has boundedly many $\bs{e}$-conjugates, then $L(b,c)$ holds.
    %  \item For any $\emptyset$-type-definable equivalence relation $L$ which is coarser than $F$ on $q'(x)$, if $b_L$ has boundedly many $\bs{e}$-conjugates, then $L(b,c)$ holds.
    \item $bc \models \exists z(z \equiv^{\KP}_{\bs{e}} x \wedge F(z,y))$.
\end{enumerate}
Therefore  it follows $\autfkpe=\aut_{\bdd({\bs e})}(\CM).$
\end{prop}

\begin{proof}
(1) $\Rightarrow$ (2):
%By  Remark \ref{bddequiv}, $\bdd(\bs{e})$ itself is interdefinable with a single hyperimaginary $\bs{e}' \in \bdd(\bs{e})$.
Notice that $x_F \equiv_{\bdd(\bs{e})}y_F$ is $\bs{e}$-invariant,  type-definable (see Remark \ref{bddequiv}) bounded
% (as in the proof of Corollary \ref{maininthesection}) 
equivalence relation coarser than $F$.
%In particular, $x_F \equiv_{\bdd(\bs{e})} y_F$ is coarser than $F$.
Hence  by Proposition \ref{propcharofKP}(2), $b_F \equiv^{\KP}_{\bs{e}} c_F$ implies $b_F \equiv_{\bdd(\bs{e})} c_F$.
\medskip

%$(2) \Leftrightarrow (3)$ :
%It is exactly Lemma \ref{lembddinterdef}.

% $(2) \Rightarrow (4)$ :
% Assume (2) and let $L$ be as in (4).
% Since $b_L \in \bdd(\bs{e})$, $L(b, f(b))$ holds for any $f \in \aut_{\bdd(\bs{e})}(\mathcal{M})$, thus $L(b, c)$ holds.
% \newline

%\medskip
(2) $\Rightarrow$ (3):
Assume (2) and the  conditions for $L$ in (3).
By Fact \ref{fklorefct}, there is a hyperimaginary $ba/L'$ such that $ba/L'$ and $b_L$ are interdefinable over $\bs{e}$.
Then $ba/L' \in \bdd(\bs{e})$, thus by (2), there is $f \in \aut_{(ba/L')\bs{e}}(\mathcal{M})$ such that $f(b_F) = c_F$.
Then $F(f(b),c)$ holds, thus $L(f(b), c)$.
But $f \in \aut_{(ba/L')\bs{e}}(\mathcal{M})$, hence fixes $b_L$, thus $L(b,f(b))$ and by transitivity $L(b,c)$ holds.

\medskip
(3) $\Rightarrow$ (1): By Proposition \ref{propcharofKP}(1),
$x_F \equiv^{\KP}_{\bs{e}} y_F$ satisfies all conditions of $L$ in (3). Thus  (3) implies that $b_F \equiv^{\KP}_{\bs{e}} c_F$.

\medskip
(1) $\Leftrightarrow$ (4):
Easy to check.

\medskip 
Now the last equality  follows from Corollary \ref{corfinarity}.
\end{proof}

We now  recall the notion of G-compactness  and its equivalent   conditions  in the context of hyperimaginaries.

\begin{Definition}
  $T$ is called \emph{G-compact} over $\bs{e}$ if the trivial identity subgroup 
  %$\{\operatorname{id} / \autfe(\mathcal{M})\}$ 
  is closed in $\gall(T,\bs{e})$.
\end{Definition}
 
 %By the general facts on topological groups [refer??],  
 Therefore by the general facts on compact groups (see \cite{HM}), $T$ is G-compact over $\bs{e}$ iff 
 $\gall(T,\bs{e})$ is Hausdorff (so compact) iff $\autf_{\bs e}(\CM)=\autfkpe$
 iff $\gall(T,\bs e)$ and $\gal_{\KP}(T,\bs e)$ are isomorphic as topological groups.

\begin{prop}\label{propcharofgcpt}
The following are equivalent.
\begin{enumerate}
    \item $T$ is G-compact over $\bs{e}$.
    \item For any hyperimaginaries $b_F$ and $c_F$, $b_F \equiv^{\KP}_{\bs{e}} c_F$ iff $b_F \equiv^{\LL}_{\bs{e}} c_F$.
    \item For any small real tuples $b$ and $c$, $b \equiv^{\KP}_{\bs{e}} c$ iff $b \equiv^{\LL}_{\bs{e}} c$.
    \item For any $\emptyset$-type-definable equivalence relation $F$, $x_F \equiv^{\LL}_{\bs{e}} y_F$ is type-definable.
\end{enumerate}
\end{prop}

\begin{proof}
(1) $\Rightarrow$ (2):
By (1), $\autfkpe = \autfe(\mathcal{M})$, thus the orbit equivalence relations $\equiv^{\autfkpe}_{\bs e}$ and $\equiv^{\autfe(\mathcal{M})}_{\bs e}$ coincide for any hyperimaginaries.
\medskip

(2) $\Rightarrow$ (3):
Trivial by letting $F$ the equality.
\medskip

(3) $\Rightarrow$ (1):
Recall that  $M \models T$  and $a\in M$. Now let  $f \in \autfkpe$.
Assuming (3), in particular we have $M \equiv^{\LL}_{\bs{e}} f(M)$, thus there is $g \in \autfe(\CM)$ such that $g(M) = f(M)$.
Then by Remark \ref{preli}(3), it follows that 
%$f / \autfe(\CM) = g / \autfe(\CM)$, implying 
$f \in \autfe(\CM)$.
Hence  $\autfkpe = \autfe(\CM)$, and (1) follows.
\medskip

(2) $\Rightarrow$ (4):
Clear by  Proposition \ref{propcharofKP}(1).
%since $x_F \equiv^{\KP}_{\bs{e}} y_F$ is type-definable by Proposition \ref{propcharofKP} (or Corollary \ref{maininthesection}).
\medskip

(4) $ \Rightarrow$  (2):
We already know that $x_F \equiv^{\LL}_{\bs{e}} y_F$ is an $\bs{e}$-invariant bounded equivalence relation  coarser than $F$ (Remark \ref{preli}(2)).
Then assuming (4), by Proposition \ref{propcharofKP}(2), we have (2).
\end{proof}

As is well-known any simple theory is G-compact over a hyperimaginary (see \cite{K}), while  non G-compact theories (over $\emptyset$) are presented in \cite{CLPZ} and  \cite{P}.  In \cite{CLPZ}, an example shows the KP-type and the Lascar type of a  finite tuple over $\emptyset$ can be  distinct; and   another shows   KP-types and Lascar types over $\emptyset$ of all finite tuples are the same while they are distinct for an infinite tuple. 
In \cite[Section 3.2]{P}, an example of G-compact theory $T$ over $\emptyset$ is given, while it is no longer G-compact after naming real parameters.

\medskip

Next we  define and characterize the (Shelah) strong type over a hyperimaginary.

\begin{Definition}$ $
\begin{enumerate}
    \item $\gall^0(T,\bs{e})$ denotes the connected component of the identity in  $\gall(T,\bs{e})$. 
    %containing the identity $\operatorname{id} / \Autfe(\mathcal{M})$.
    \item $\Autfse := \pi^{-1}(\gall^0(T,\bs{e}))$.
    \item Two hyperimaginaries $b_F$ and $c_F$ are said to have the same {\em (Shelah) strong type}  if there is $f \in \Autfse$ such that $f(b_F) = c_F$, denoted by $b_F \equiv^{\ss}_{\bs{e}} c_F$.  Since $\autfkpe \leq \autfse$,  $b_F \equiv^{\KP}_{\bs{e}} c_F$ implies  $b_F \equiv^{\ss}_{\bs{e}} c_F$.
\end{enumerate}
\end{Definition}

\begin{remark}\label{remstrongtpdef}
Note that $\gall^0(T,\bs{e})$ is a normal closed subgroup of $\gall(T,\bs{e})$ \cite{HM} and $\equiv^{\ss}_{\bs{e}}$ is the orbit equivalence relation $\equiv^{\autfse}_{\bs e}$, so $\equiv^{\ss}_{\bs{e}}$ is type-definable over $\bs{e}$ by Corollary \ref{maininthesection}.
We denote
$$\gal_{\SS}(T,\bs e):=\gall(T,\bs e)/\gall^0(T,\bs e)=\aut_{\bs e}(\CM)/\autfse.$$
Thus $\gal_{\SS}(T,\bs e)$ is a profinite (i.e. compact and totally disconnected) topological group.
Notice that $\gall^0(T,\bs{e})$ is the intersection of all closed (normal) subgroups of finite indices in $\gall(T,\bs e)$, since such an intersection is the identity for a profinite group \cite{HM}.
\end{remark}

% \begin{remark}\label{remstrongtpdef}
% Note that $\gall^0(T,\bs{e})$ is a normal closed subgroup of $\gall(T,\bs{e})$. Indeed $\gall^0(T,\bs{e})$ is the intersection of all closed (normal) subgroups of finite indices in $\gall(T,\bs e)$.
% Note also that $\equiv^{\ss}_{\bs{e}}$ is the orbit equivalence relation $\equiv^{\autfse}_{\bs e}$, so $\equiv^{\ss}_{\bs{e}}$ is type-definable over $\bs{e}$ by Corollary \ref{maininthesection}. We denote
% $$\gal_{\SS}(T,\bs e):=\gall(T,\bs e)/\gall^0(T,\bs e)=\aut_{\bs e}(\CM)/\autfse.$$
%  Thus $\gal_{\SS}(T,\bs e)$ is a profinite (i.e. compact and totally disconnected) topological group.
% \end{remark}

Now  in order to characterize $\equiv^{\ss}_{\bs{e}}$ in the context of hyperimaginaries, we  define the algebraic closure of $\bs e$.

 %correcting errors in \cite[Proposition 5.1.14]{K}.
%Some arguments are from \cite[Section 5.1]{K} and \cite[Section 6]{Z}.

\begin{Remark/defn}\label{lemaclinterdef}$ $
\begin{enumerate}
    \item For a hyperimaginary $\bs{e}'$, denote $\bs{e}' \in \acl(\bs{e})$ and say $\bs{e}'$ is \emph{algebraic over} $\bs{e}$ if $\{ f(\bs e')\mid f \in \aut_{\bs e}(\CM)\}$ is finite.
    \item As in Remark \ref{bddequiv}, the \emph{algebraic closure} of $\bs{e}$, $\acl(\bs{e})$ can be regarded as a bounded set of countable hyperimaginaries, which is interdefinable with 
    a single hyperimaginary $b_F \in \bdd(\bs{e})$ (but possibly $b_F \notin \acl(\bs{e})$).
 
 %   This implies that  $\pi^{-1}(\pi(\aut_{\acl(\bs{e})}(\CM))) = \aut_{\acl(\bs{e})}(\CM)$, by Proposition \ref{PropHcld}.
    %  \item Proposition \ref{propKPTFAE} says  $\autfkpe=\aut_{\bdd({\bs e})}(\CM)\leq \aut_{\acl({\bs e})}(\CM)$. 
      %Hence we have  $\pi^{-1}(\pi(\aut_{\acl(\bs{e})}(\CM))) = \aut_{\acl(\bs{e})}(\CM)$.
   
    \item Notice that given $d_i/L_i\in \acl(\bs e)$  $( i\leq n)$,
as pointed out at the end of the first paragraph of Section 1, $(d_0/L_0,\dots, d_n/L_n)$ is equivalent to a single $d_L\in \acl(\bs e)$.
Hence  by compactness,  for any hyperimaginaries $b_F$ and $c_F$,
$$b_F \equiv_{\acl(\bs{e})} c_F\text{ iff }b_F \equiv_{d_L} c_F\text{ for any }d_L \in \acl(\bs{e}).$$      
 %   \item By compactness (and (2) or Remark \ref{bddequiv}), for any hyperimaginaries $b_F$ and $c_F$,
    %$$b_F \equiv_{\acl(\bs{e})} c_F\text{ iff }b_F \equiv_{d_L} c_F\text{ for any }d_L \in \acl(\bs{e}).$$
\end{enumerate}
\end{Remark/defn}

The following clarifies Propositions 5.1.14 and 5.1.17 in \cite{K}.  There, the proof of 5.1.14(1)$\Rightarrow$(2) with the hyperimaginary parameter    need not work, and (1),(5)  should be deleted out (we will supply a counterexample). 

\begin{prop}\label{propcharofstr} $ $
\be \item
 $\autfse=\aut_{\acl({\bs e})}(\CM).$
\item Let $b_F, c_F$ be hyperimaginaries.
The following are equivalent.
\begin{enumerate}
    \item $b_F \equiv^{\ss}_{\bs{e}} c_F$.
    \item $b_F \equiv_{\acl(\bs{e})} c_F$.
    %\item $b_F \equiv_{d_L} c_F\text{ for any }d_L \in \acl(\bs{e}).$
    \item For any $\bs{e}$-invariant type-definable equivalence relation $L$ coarser than $F$, if $b_L$ has finitely many conjugates over $\bs{e}$, then $L(b,c)$ holds.
    \item $bc \models \exists z(z \equiv^{\ss}_{\bs{e}} x \wedge F(z,y))$.
\end{enumerate}
\ee
%Therefore  it follows
\end{prop}

\begin{proof}
(1)  We claim first that 
$$\gall^0(T,\bs{e}) =\bigcap \{\pi(\aut_{d_L \bs{e}}(\mathcal{M}))\mid d_L\in\acl(\bs e)\}:$$ 

Let $d_L \in \acl(\bs{e})$ where $d_L$ is a hyperimaginary. Let $d^0_L(=d_L), \dots, d^n_L$ be all the conjugates of $d_L$ over $\bs e$. Then any $f\in \aut_{\bs e}(\CM)$ permutes the set $\{ d^0_L, \dots, d^n_L\}$. Hence it follows that 
$\aut_{d_L \bs{e}}(\mathcal{M})$ has a finite index in $\aut_{ \bs{e}}(\mathcal{M})$. Thus (due to Proposition \ref{PropHcld}) $\aut_{d_L \bs{e}}(\mathcal{M})=\pi^{-1}(\pi(\aut_{d_L \bs{e}}(\mathcal{M})))$ and 
  $\pi(\aut_{d_L \bs{e}}(\mathcal{M}))$ is a closed subgroup of  finite index in $\gall(T,\bs{e})$. 
Then as  in  Remark \ref{remstrongtpdef}, we have $\gall^0(T,\bs{e}) \leq \pi(\aut_{d_L \bs{e}}(\mathcal{M}))$.

Conversely, given a normal closed subgroup $H' \leq \gall(T,\bs{e})$ of finite index,
and  $H := \pi^{-1}(H')$,  Proposition \ref{PropHcld} says $H' = \pi(\aut_{b_F \bs{e}}(\mathcal{M}))$ for some $b_F \in \bdd(\bs{e})$. But since $H'$ is of finite index,  so is $H=\aut_{b_F \bs{e}}(\mathcal{M})$   in $\aut_{\bs{e}}(\mathcal{M})$,  and we must have $b_F \in \acl(\bs{e})$. Thus the claim follows from Remark   \ref{remstrongtpdef}.

Therefore 
\begin{align*}
\autfse & = \pi^{-1}(\gall^0(T,\bs{e}))
 =\pi^{-1}(\bigcap\{ \pi( \aut_{d_L\bs{e}}(\mathcal{M}))\mid
d_L\in \acl(\bs e) \})\\
&=\bigcap  \{\aut_{d_L\bs{e}}(\mathcal{M})\mid
d_L\in \acl(\bs e) \}
 = \aut_{\acl(\bs{e})}(\mathcal{M}),
\end{align*}
 where the last equality follows by Remark \ref{lemaclinterdef}(3).
\medskip

(2)(a) $ \Leftrightarrow $ (b): It follows from (1).

\medskip

(b) $\Rightarrow$ (c):
Suppose that $L$ is an $\bs{e}$-invariant type-definable equivalence relation coarser than $F$, and $b_L$ has finitely many conjugates over $\bs{e}$.
By Fact \ref{fklorefct}, there is a hyperimaginary $ba/L'$ such that $ba/L'$ and $b_L$ are interdefinable over $\bs{e}$.
Then $ba/L' \in \acl(\bs{e})$. Now by (b), there is  $f \in \aut_{\acl(\bs{e})}(\mathcal{M})$ such that  $f(b_F) = c_F$, so  $F(f(b), c)$. This $f$ fixes $ba/L'$ and $\bs{e}$, in turn, $f$ fixes $b_L$.
Thus $L(f(b), b)$, and since $L$ is coarser than $F$, we have $L(b,c)$.
\medskip

(c) $\Rightarrow$ (b): Assume (c). 
Let a hyperimaginary $d_L \in \acl(\bs{e})(\subseteq \bdd(\bs{e}))$ and let $\{d_i/L : i \in I\}$ be the set of all $\bs{e}$-conjugates of $d_L$, where $I$ is a finite set and  $d_i\equiv_{\bs e} d$ for each $i\in I$. Due to Remark \ref{lemaclinterdef}(3), it suffices to show $b_F \equiv_{d_L} c_F$.

Let $p(x,d):=\tp(b_F/d_L)= \exists z_1z_2(\tp_{z_1z_2}( bd )\wedge F(x,z_1) \wedge L(z_2, d)).$
%which can be called the complete type of $b_F$ over $d_L\bs{e}$ (its solution set is the set of all $(d_L\bs{e})$-conjugates of $b_F$).
%Clearly $p(x,d_i)=\tp(b_F/d_i/L)$.
Put $J$ the maximal subset of $I$ such that $(p(x,d) \in) \{p(x,d_i) : i \in J\}$ is realized by $b$.
Then denote $d_J = (d_i : i \in J)$, $w_J = (w_i : i \in J)$. 
%We let $q(x,w,a)=\tp(b_Fd_L/\bs e)$, 
and let 
\begin{align*}
p'(x,w_J,a) & :=\tp_{w_J}(d_J/\bs e) \wedge \bigwedge_{j\in J} p(x,w_i) \\
& = \exists z(w_Jz\equiv d_Ja \wedge E(z,a)) \wedge \bigwedge_{j\in J} \exists y_iz_i\equiv bd( F(x, y_i)\wedge L(z_i,w_i)) .
\end{align*}
%$p'(x,w_J,a)=\tp(b_F(d_i/L : i \in J)/\bs e)$. Recall again
%\[p'(x,w_J;a) \leftrightarrow \exists z_1v_Jz_2(z_1v_Jz_2 \equiv bd_Ja \wedge F(x,z_1) \wedge \bigwedge_{i \in J}L(w_i, v_i) \wedge E(z_2, a)).\]
%the complete type of $b_F(d_i/L : i \in J)$ over $\bs{e}$.

\begin{Claim}
\[L'(x,y;a) := \exists w_J(p'(x,w_J;a) \wedge p'(y,w_J;a)) \  \vee \ F(x,y)\]
is an $\bs{e}$-invariant type-definable equivalence relation coarser than $F$.
\end{Claim}

\begin{proof}[Proof of  Claim]
We check $L'$ is an equivalence relation:
It is clearly reflexive and symmetric.
For transitivity, say $L'(b_1,b_2)$ and $L'(b_2,b_3)$ holds.
\begin{enumerate}
    \item[(i)] If $F(b_1,b_2)$ and $F(b_2,b_3)$, then $F(b_1,b_3)$, so nothing to prove.
    \item[(ii)] If $F(b_1,b_2)$ and there is $u_J$ such that $b_2b_3u_J \models p'(x,w_J,a) \wedge p'(y,w_J,a)$, since $b_1/F = b_2/F$, we have $b_1b_3u_J \models p'(x,w_J,a) \wedge p'(y,w_J,a)$.
    \item[(iii)] If there are $u_J,v_J$ such that $b_1b_2u_J \models p'(x,w_J,a) \wedge p'(y,w_J,a)$ and $b_2b_3v_J \models p'(x,w_J,a) \wedge p'(y,w_J,a)$, then in particular $b_2u_J$ and $b_2v_J$ both satisfy $p'(x,w_J,a)$.
   But we must have  $\{u_i/L\mid i\in J \}= \{v_i/L\mid i\in J \}$ as sets,  since if  $\{u_i/L\mid i\in J \} \subsetneq \{u_i/L, 
\ v_i/L\mid i\in J \} $, it contradicts the maximality of $J$.  Then due to the definition of $p'$, we  have $b_1b_3u_J \models p'(x,w_J,a) \wedge p'(y,w_J,a)$.
\end{enumerate}
$L'$ is clearly coarser than $F$ and since $p'$ is $\bs{e}$-invariant, $L'$ is also $\bs{e}$-invariant. Claim is proved.
\end{proof}
%$L'(x,y)$ in the Claim satisfies the conditions of (3); the only one we need to check is whether 
We now check that $b_{L'}$ has finitely many $\bs{e}$-conjugates.
Notice that $(d_i/L : i \in J)$ has finitely many $\bs{e}$-conjugates, say the number is $m_0$.
But by the definition of $L'$, $b_{L'}$ cannot have more than $m_0$-many conjugates.

Therefore  by  Claim and (c), we have  $L'(b,c)$. Thus $F(b,c)$ or $bc \models \exists w_J(p'(x,w_J,a) \wedge p'(y,w_J,a))$. The former case, of course $b_F \equiv_{d_L} c_F$.
%If $bc \models \exists w_J(p'(x,w_J,a) \wedge p'(y,w_J,a))$, 
The latter case, say there is $d'_J$ such that $p'(b,d'_J,a)$ and  $p'(c,d'_J,a)$ hold. Note that
we already have $ p'(b,d_J,a)$, thus as  argued in (iii) above,  we must have $\{d_i/L\mid i\in J\}=\{d'_i/L\mid i\in J\}$ and   $ \models p'(c,d_J,a)$.
But since $p(x,d) \in \{p(x,d_i) : i \in J\}$, we have $c \models p(x,d)$ and so $b_F \equiv_{d_L} c_F$, as wanted.

\medskip

 (a) $\Leftrightarrow$ (d):
Easy to check.
%\medskip 
%Now the last equality  follows from Corollary \ref{corfinarity}.
\end{proof}

We now state  several   criterion for  KP-types and strong types over $\bs e$ being equal, reconfirming \cite[5.1.18]{K}. Note that, $\gal^0_L(T,\bs e)=\ov{\{\id\}}$ iff 
 $\gal_{\KP}(T,\bs{e})$ is profinite  iff $\autfkpe(\CM)=\autfse$
 iff $\gal_{\KP}(T,\bs e)$ and $\gal_{\SS}(T,\bs e)$ are isomorphic as topological groups \cite{HM}.

\begin{prop}
The following are equivalent.
\begin{enumerate}
    \item $\autfkpe = \autfse$.
    \item For any hyperimaginaries $b_F$ and $c_F$, $b_F \equiv^{\KP}_{\bs{e}} c_F$ iff $b_F \equiv^{\ss}_{\bs{e}} c_F$.
    \item For any real small tuples $b$ and $c$, $b \equiv^{\KP}_{\bs{e}} c$ iff $b \equiv^{\ss}_{\bs{e}} c$.
    \item For any real finite tuples $b$ and $c$, $b \equiv^{\KP}_{\bs{e}} c$ iff $b \equiv^{\ss}_{\bs{e}} c$.
   % \item $\operatorname{Gal}_{\KP}(T,\bs{e}) = \aute(\CM) / \autfkpe$ is compact Hausdorff totally disconnected.
    \item $\acl(\bs{e})$ and $\bdd(\bs{e})$ are interdefinable.
\end{enumerate}
\end{prop}

\begin{proof}
(1) $\Rightarrow$ (2):
Directly follows by the definition of $\equiv^{\KP}_{\bs{e}}$ and $\equiv^{\ss}_{\bs{e}}$.
\medskip 

(2) $\Rightarrow$ (3):
Clear by letting $F$ the equality.
\medskip 

(3) $\Rightarrow$ (1): Use 
the same method as the proof of $(3) \Rightarrow (1)$ in Proposition \ref{propcharofgcpt}, and work with 
  $M \models T$  such that  $\bs e\in \dcl(M)$. Let  $f \in \autfse$.
Then $M \equiv^{\ss}_{\bs{e}} f(M)$, thus by (3), there is $g \in \autfkpe$ such that $g(M) = f(M)$.
By Remark \ref{preli}(3), there is $h \in \autfe(\CM)\leq \autfKPe$ such that $hg = f$, thus $f \in \autfkpe$.
Hence we  get $\autfkpe = \autfse$.
\medskip 

(3) $\Rightarrow$ (4):
Trivial.
\medskip 

(4) $\Rightarrow$ (3):
Let $b$ and $c$ be any real small tuples such that $b \equiv^{\ss}_{\bs{e}} c$ and let  $b_0, c_0$ be any corresponding finite tuples respectively.
Then by (4), $b_0 \equiv^{\KP}_{\bs{e}} c_0$. Thus by
%the type-definability of $x \equiv^{\KP}_{\bs{e}} y$ (Proposition \ref{propcharofKP}) and compactness (or simply by 
Corollary \ref{corfinarity}, we have $b \equiv^{\KP}_{\bs{e}} c$.
\medskip 
%$(1) \Leftrightarrow (5)$ :
%$\aute(\CM) / \autfkpe \cong \gall(T, \bs{e}) / \pi(\autfkpe)$ and a topological group is compact Hausdorff totally disconnected iff its identity component is trivial.

(1) $\Leftrightarrow$ (5):  This comes from $\autfkpe=\aut_{\bdd({\bs e})}(\CM)$ (Proposition \ref{propKPTFAE}), and  $\autfse=\aut_{\acl({\bs e})}(\CM)$ (Proposition \ref{propcharofstr}).
\end{proof}

We finish this section with an investigation of $\acleq$ and $\acl$.
Recall that  $\acleq(\bs{e}) := \{\bs{e}\} \cup (\acl(\bs{e}) \cap \CM^{\operatorname{eq}})$ is 
the \emph{eq-algebraic closure} of $\bs{e}$, where as usual $\CM^{\operatorname{eq}}$ is the set of all imaginary elements (equivalence classes of $\emptyset$-definable equivalence relations) of $\CM$.
%The following fact follows from \cite[Theorem 21(2)]{Z} (with Remark \ref{preli}(2) or  \cite[Lemma 20]{Z}).

\begin{remark}\label{coraclacleqeq}
We point out that in any $T$, $\acl(\emptyset)$ and $\acleq(\emptyset)$ are interdefinable over $\emptyset$. This follows from
 Proposition \ref{propcharofstr}(1) and the fact that $\aut_{\acleq(\emptyset)}(\CM) = \autf_{\SS}(\CM)$ \cite[Theorem 21(2)]{Z}, or one can directly
 show it  by compactness: If a hyperimaginary $c/F$ has finitely many conjugates  over $\emptyset$, then 
the union of its conjugate classes is $\emptyset$-type-defined, say by $\Phi(x)$. In $\Phi(x)$, $F$ is relatively $\emptyset$-definable with finitely many classes.
Hence by compactness one can find  a formula $\delta(x)\in\Phi(x)$ and an $\emptyset$-definable finite equivalence relation $F'$ on $\delta(x)$ such that
$c/F$ and the imaginary $c/F'$ are interdefinable. 
 
 %say $c_0/F,\dots, c_n/F$,
 %then $F(x,c_0)\vee\dots \vee F(x,c_n)$ is 
\end{remark}

%\begin{fact}\label{factacleqstr}
%For any theory $T$, $\aut_{\acleq(\emptyset)}(\CM) = \autf_{\SS}(\CM)$.
%\end{fact}

%\begin{Corollary}\label{coraclacleqeq}
%$\acl(\emptyset)$ and $\acleq(\emptyset)$ are interdefinable over $\emptyset$.
%\end{Corollary}

%\begin{proof}
%By Fact \ref{factacleqstr} and Proposition \ref{propcharofstr}.
%\end{proof}

However, contrary to \cite[Corollary 5.1.15]{K},  in general $\acl(\bs e)$ and $\acl^{\eq}(\bs e)$ {\em need not} be interdefinable. 
As said before Proposition \ref{propcharofstr},  the error occurred there due to the incorrect proof of \cite[5.1.14(1)$\Rightarrow$(2)]{K}.
An example  presented in \cite{DKKL} for another purpose  supplies a counterexample.
Consider the following 2-sorted model: 
%\newline
%\newline

\medskip

$M$ = $((M_1$, $S_1$, $\{g^1_{1/n} \mid n \geq 1\})$, $(M_2$, $S_2$, $\{g^2_{1/n} \mid n \geq 1\})$, $\delta)$ where

\begin{enumerate}
    \item $M_1$ and $M_2$ are  unit circles centered at origins of two disjoint  (real) planes.
    \item $S_i$ is a ternary relation on $M_i$, defined by $S_i(b,c,d)$ holds iff $b$, $c$ and $d$ are in clockwise-order.
    \item $g^i_{1/n}$ is a unary function on $M_i$ such that $g^i_{1/n}(b)$ = rotation of $b$ by ${2\pi}/n$-radians clockwise.
    \item $\delta : M_1\rightarrow M_2$ is the double covering, i.e.
    $\delta(\text{cos $t$, sin $t$})$ = $(\text{cos $2t$, sin $2t$})$.
    \item Let $\CM$ be a monster model of Th($M$), and let  $\CM_1$ and $\CM_2$ be the two sorts of $\CM$.
    %which are monster models of Th$(M_i$, $S_i$, $\{g^i_{1/n} \mid n\geq 1 \})$, $i$ = 1, 2.
\end{enumerate}

In  \cite[Theorems 5.8 and  5.9]{DKL}, it is shown that  $\Th(\CM_i)$ has 
  weak elimination of imaginaries (that is, for  any imaginary element $c$ there is a real tuple $b$ such that $c\in\dcl(b)$ and 
  $b\in \acl(c)$), and quantifier elimination. 
We point out that  similar proofs, which we omit,  yield
the same results. 
%that  
 %$\operatorname{Th}(\CM)$ has quantifier elimination and weal elimination of imaginaries.  

\begin{Fact}\label{prophasqe}
$\operatorname{Th}(\CM)$ has quantifier elimination and weak elimination of imaginaries.
\end{Fact}

Now for $i=1,2$, we let $E_i(x,y)$ iff $x$ and $y$ in $\CM_i$ are infinitesimally close, i.e.
\[E_i(x,y) := \bigwedge_{1<n} (S_i(x,y,g^i_{1/n}(x)) \vee S_i(y,x,g^i_{1/n}(y))),\]
which is an $\emptyset$-type-definable equivalence relation.
Let $c \in \CM_2$, $c_1, c_2 \in \CM_1$ where $\delta(c_1) = \delta(c_2) = c$, so that $c_1,c_2$ are antipodal to each other.
Notice that $c_1/E_1$ and $c_2/E_1$ are conjugates over  $c/E_2$, so that  $c_1/E_1, c_2/E_1 \in \acl(c/E_2)$.
Then the following implies that  $\acl(c/E_2)$ and $\acl^{\eq}(c/E_2)$ are {\em not} interdefinable. 

\begin{Claim}
$\acl^{\eq}(c/E_2)$ and $c/E_2$ are interdefinable in $\CM$.
\end{Claim}
\begin{proof}
To lead a contradiction suppose that there are distinct  imaginaries $d_1,d_2\in \acl^{\eq}(c/E_2)$ such that $d_1\equiv_{c/E_2} d_2$. 
Now     weak elimination of imaginaries of $\Th(\CM)$ implies that 
$\acl^{\eq}(d_1,d_2)$ and $B:= \{b\in \CM\mid b\in \acl^{\eq}(d_1,d_2)\}$ are interdefinable \ (*). In particular,  $B\subseteq  \acl^{\eq}(c/E_2)\cap \CM$.
However, 
due to quantifier elimination,  for any  infinitesimally close $b, b'\in \CM_i$ ($i=1,2$), there is $f \in \aut_{c/E_2}(\CM)$ sending $b$ to $b'$. Hence 
indeed $B=\emptyset$, which contradicts  (*).
\end{proof}

\section{Diameter}

In \cite{N}, Newelski showed 
  that a type-definable Lascar strong type over a real set $A$ has a finite diameter, which implies that  $T$ is G-compact over $A$ iff  there exists a uniform $n < \omega$ bounding  the diameters of any real tuples over $A$. 
The proof uses intricate analyses of   open subsets  of the type-definable set. Then 
 in \cite{P},  using the notion of {\em c-free sets}, a shorter  and direct proof showing the same result is given.
 Here by applying  arguments  in  \cite{P}, we prove that the same holds  for  hyperimaginaries.
 We cite basic definitions below.
 \medskip
 
 Let $\Phi(x)=\Phi(x,b)$ be a partial type over $b$. Assume that   for any $c,d\models \Phi(x)$ there is $f\in \aut(\CM)$ such that 
 $f(c)=d$ while $\Phi(x,b) \leftrightarrow \Phi(x, f(b))$, i.e. $\Phi(\CM)=f(\Phi(\CM))$ \  (*). Note that any two realizations of $\Phi(x)$ have the same type over
 $\emptyset$.

 \begin{Definition}$ $
\be
\item A formula $\psi(x,d)$ is said to be {\em c-free} over $\Phi$ if  there are  $f_0, \dots, f_{n} \in \aut(\CM)$ for some $n<\omega$ such that 
$$\Phi(x) \models \bigvee_{i\leq n}\psi(x,f_i(d))$$ 
and $f_i(\Phi(\CM)) = \Phi(\CM)$ for every $i \leq n$.

   \item  A  type $q(x)$ is said to be {\em c-free} over $\Phi$ if for any $\varphi(x)\in \CL(\CM)$ such that $q(x) \vdash \varphi(x)$, $\varphi(x)$ is c-free over $\Phi$.
   
   \item   A  type $q(x)$ is said to be {\em weakly c-free} over $\Phi$ if for any $\psi(x)\in \CL(\CM)$ with $q(x) \vdash \psi(x)$, there is a non-c-free formula
   $\phi(x)\in\CL(\CM)$ over $\Phi$ such that $\psi(x)\vee \phi(x)$ is c-free over $\Phi$.
\ee
\end{Definition}

\begin{Definition} \label{wcfdef}
     We put $$P^{\Phi} : = \{q(x,y) \in S(\emptyset) \mid  q(x,y) \cup \Phi(x) \cup \Phi(y)\mbox{ is consistent}\},$$
    and
     $$P^{\Phi}_{\wcf}: = \{q(x,y) \in P^{\Phi} \mid q(d,y)\mbox{ is weakly c-free over }\Phi\mbox{ for any (some) }d \vDash \Phi \}.$$
\end{Definition}

%The proofs of the following Remark and Proposition are given in \cite{P}, Remark 3.3.5 and Proposition 3.3.6.

The following is the key result from \cite{P} which we will use.

\begin{Fact}\label{Diamkeyprop}  $ $
\be\item $P^{\Phi}$ and $P^{\Phi}_{\wcf}$ are closed and non-empty in $S_{xy}(\emptyset)$.
\item
Let  $D\ne \emptyset$ be a relatively open subset  of $P^{\Phi}_{wcf}$.
Then there are $d_0, \dots, d_k \vDash \Phi$ such that for any $d \vDash \Phi$, there is $d' \vDash \Phi$ holding
\be
    \item $\tp(d,d') \in D$ and 
    \item $\tp(d_i, d') \in D$ for some $i \leq k$.
    \ee
\ee
\end{Fact}

Now we begin to prove our hyperimaginary context result Theorem \ref{typedeflstp}.
Assume that given an $\emptyset$-type-definable equivalence relation $F$ and real $b$,  the Lascar strong type of $b_F$ over $\bs{e}$ is type-definable. That is 
there is a partial type $\pi_0(x)$  (over $ab$) such that $d\models \pi_0(x)$ iff $d_F \equiv_{\bs{e}}^{\LL} b_F$.
% where 
%, $b,d\in \CM$, and 
%$\bs{e} = a / E$ is a hyperimaginary, i.e. 
%We recall the notion of Lascar diameter.

%\medskip

\begin{Definition}$ $
%\begin{enumerate}
    %\item 
    Recall that we say  the {\em Lascar distance} between $b_F$ and $d_F$ over $\bs e$ is at most $n \in \omega$, denoted by $d_{\bs{e}}(b_F,d_F) \leq n$, if
    % Recall that   the {\em Lascar distance} (or {\em diameter}) between $b_F$ and $d_F$ over $\bs e$ is at most $(1\leq) n\in \omega$, denoted by $d_{\bs{e}}(b_F,d_F) \leq n$, if  
    %$b_F=d_F$ when $n=0$; or for $n>0$,
    there are $b=b_0, b_1, \dots, b_{n}=d\in \CM$ 
    such that $b_i/F$,  $b_{i+1}/F$ begin an $\bs e$-indiscernible sequence for each $i < n$.
    If $d_{\bs e}(b_F, d_F)\leq n$ for some $n<\omega$, then we let  $d_{\bs e}(b_F, d_F) := \min \{m <\omega\mid  d_{\bs e}(b_F, d_F)\leq m\}$.
    The \emph{(Lascar) diameter} of $b_F$ over $\bs e$ is $\max\{d_{\bs{e}}(b_F,d_F) \mid b_F \equiv^{\LL}_{\bs e} d_F\}$ if exists.
    Otherwise, it is $\infty$.
\end{Definition}

\begin{Remark}\label{diantypedef}$ $
\be\item
 By Remark \ref{preli}(2), $b_F \equiv_{\bs{e}}^{\LL} d_F$ iff $d_{\bs{e}}(b_F,d_F) \leq n$ for some $n < \omega$.

\item For $r(x,y):=\tp(ba)$, we let
$$\pi(x,y) \equiv \pi_0(x) \wedge r(x,y) \wedge E(y,a),$$
 a type over $ab$.
If $dc, d'c' \vDash \pi(x,y)$, then there is $f \in \aut(\CM)$ such that $f(dc) = d'c'$. Then since  $f(\bs e)=\bs e$ and $f(\pi_0(\CM))=\pi_0(\CM)$, it 
follows that  $f(\pi(\CM)) = \pi(\CM)$.
 Therefore $\pi(x,y)$ satisfies   the assumption (*)  in the beginning of this section.
Now as in Definition \ref{wcfdef}, we let 
$$P(xy,x'y'):=P^\pi \mbox{ and }  P_{\wcf}(xy,x'y'):=P^\pi_{\wcf},$$
and we can apply Fact  \ref{Diamkeyprop}  to $P$ and $P_{\wcf}$.

\item Recall that   $(b_i/F\mid i<\omega)$ is $\bs e$-indiscernible iff there are $a'\equiv a$ with $E(a,a')$  and $b'_i$ ($i<\omega$) with $F(b_i,b'_i)$ such that 
$(b'_i\mid i<\omega)$ is $a'$-indiscernible.

\item Clearly there is  a type $\Lambda(x_0,x_1,z)$ over $\emptyset$ saying $x_0,x_1$ begin a $z$-indiscernible sequence.
Using (3) and $\Lambda$, we can $\emptyset$-type-define a relation $\Phi_n(xy, x'y')$ saying that $(\tp_{xy}(ba) = )r(x,y), r(x',y'),E(y,y')$ and $d_{y/E}(x_F,x'_F)\leq n$. Namely we let 
\begin{align*}
\Phi_n(xy;x'y') \equiv & \ r(x,y) \wedge r(x',y') \wedge E(y,y') \wedge \\
%&\wedge  \exists x_0x'_0 \dots x_{n}x'_n, z_0 \dots z_{n-1}(x=x_0 \ \wedge \ x_n'=x'\wedge \bigwedge_{i< n}E(y,z_i) \wedge \Lambda (x'_i, x_{i+1},z_i)).\\
%& \exists (x_ix'_{i+1} z_i)_{i< n}  (F(x,x_0)  \wedge  F(x'_n,x')\wedge \bigwedge_{i< n}F(x_i,x'_i)\wedge E(y,z_i) \wedge \Lambda (x_i, x'_{i+1},z_i)).\\
& \exists x_n \exists(x_ix'_{i} z_i)_{i< n} (x=x_0  \wedge  F(x_n,x')\wedge \\ 
& \bigwedge_{i< n}F(x_i,x'_i)\wedge E(y,z_i) \wedge \Lambda (x'_i, x_{i+1},z_i)).
\end{align*}
\ee
\end{Remark}

% The following is a result of the Baire Category Theorem.

\begin{fact}[Baire Category Theorem]\label{BCT} $ $
\begin{enumerate}
    \item A topological space $X$ is a {\em Baire space} iff whenever the union of countably many closed subsets of $X$ has an interior point, at least one of the closed subsets must have an interior point.
    \item Any locally compact Hausdorff space is a Baire space.
\end{enumerate}
\end{fact}

We are ready to prove our goal.

\begin{Theorem} \label{typedeflstp}
Type-definable Lascar strong type of a hyperimaginary  over a hyperimaginary has a finite diameter.
\end{Theorem}

\begin{proof} We will show there is  $m \in \omega$ such that for any $u_F$  with $u_F \equiv^{\LL}_{\bs{e}} b_F$, we have $d_{\bs{e}}(u_F,b_F) \leq m$ 
\ ($\dag$).
As pointed out in Fact \ref{Diamkeyprop}(1), $P$ is closed in $S_{xy,x'y'}(\emptyset)$. Indeed, $P(xy,x'y')$ is type-defined by 
$$\exists zwz'w'(xyx'y'\equiv zwz'w' \wedge \pi(z,w)\wedge \pi(z',w')),$$
which is $\emptyset$-invariant, so $\emptyset$-type-definable. 

\begin{Claim}
$P (xy,x'y') \subseteq \bigcup_{i<\omega} [\Phi_i(xy,x'y')]$.
\end{Claim}

\begin{proof}[Proof of Claim]
Suppose that  $\tp(d_0c_0d'_0c'_0) \in P$. Hence there is $dcd'c'\equiv d_0c_0d'_0c'_0$ such that  $\pi(d,c)$ and  $ \pi(d',c')$ hold.
It suffices to show that 
$ \Phi_n(dc,d'c')$ for some $n$.
%It suffices to show that $dcd'c' \vDash \Phi_n(xy,x'y')$ for some $n < \omega$.
Notice that $dc \equiv ba \equiv d'c'$ and $c_E =  \bs{e}(= a_E) = c'_E$ hold.
Also, since $d_F \equiv^{\LL}_{\bs{e}} d'_F(\equiv^{\LL}_{\bs e} b_F)$, 
$ \Phi_n(dc,d'c')$ must holds for some $n$ as described in Remark \ref{diantypedef}.
\end{proof}

Now we let $X_i = P  \cap [\Phi_i(xy,x'y')]\subseteq S_{xyx'y'}(\emptyset)$, so that
 $P = \bigcup_{i<\omega} X_i  $ by Claim.
Note that  both $X_i$ and $P_{\wcf}$ are closed (Fact \ref{Diamkeyprop}),  so $P_{\wcf} = \bigcup_{i<\omega} (P_{\wcf} \cap X_i)$ is compact Hausdorff.
Then by Baire Category Theorem, there is $n < \omega$ such that the interior of $P_{\wcf} \cap X_n$ in $P_{\wcf}$ is nonempty, so we can apply Fact \ref{Diamkeyprop}.
Namely there are $u_0v_0, \dots, u_kv_k \vDash \pi$ 
such that for any real $uv\models \pi$, there is  $u'v'\models \pi$ such that 
\be
    \item[(a)] $\tp(uv,u'v') \in X_n$ and 
    \item[(b)]  $\tp(u_{i_0}v_{i_0}, u'v') \in X_n$ for some $i_0 \leq k$.
    \ee
In particular, $\bs e=v_E=v'_E=v_i/E$ for all $i\leq k$, and $d_{\bs e}(u_F, u'_F), d_{\bs e}(u_{i_0}/F, u'_F)\leq n$  \ ($*$). 
Now  there is  $m_0\in \omega$ such that 
%$$n_{ij} = \min \{l < \omega \mid  d_{\bs{e}}(u_i/F,u_j/F) \leq l\}.$$
%Then we let $m_0=\max\{n_{ij}\mid i,j\leq k\}$.
$d_{\bs{e}}(u_i/F,u_j/F) \leq m_0$ for all $i,j\leq k$,  since $u_i/F\equiv^{\LL}_{\bs e} u_j/F$.
%Then we let $m_0=\max\{n_{ij}\mid i,j\leq k\}$.
We  show that  $m=4n+m_0$ is the bound  in ($\dag$) above.

Suppose that $u_F \equiv^{\LL}_{\bs{e}} b_F$. 
%then $d_{\bs e}$. 
Since there is $f\in\autf_{\bs e}(\CM)$ such that  $u_F=f(b_F)$, without loss of generality we can assume $u=f(b)$. Hence there is  $v$ such that $uv\equiv ba$ and 
 $\pi(u,v)$ hold. Since $\pi(b,a)$ holds as well, by (a)(b) above,  there are  $u'v', b'a' \vDash \pi$ such that 
 
\begin{enumerate}
    \item $\tp(ba,b'a'), \tp(uv,u'v') \in X_n$, and 
    \item $\tp(u_j v_j,b'a'), \tp(u_l v_l ,u'v') \in X_n$ for some $ j, l \leq k$.
\end{enumerate}

%Notice that $\tp(bc,b'c'), \tp(dc_d,d'c_d') \in X_n$ implies $d_{\bs{e}}(b,b'), d_{\bs{e}}(d,d') \leq N$.
%Let $m_0 = \min \{n < \omega \mid  d_{\bs{e}}(u_i/F,u_j/F) \leq n$ for all $i,j \leq k\}$.
Therefore by the same reason in $(*)$ above, 

\begin{align*}
d_{\bs{e}}(u_F, b_F) &\leq d_{\bs{e}}(u_F, u'_F) + d_{\bs{e}}(u'_F, u_l/F) + d_{\bs{e}}(u_l/F, u_j/F)\\
& \quad + d_{\bs{e}}(u_j/F, b'_F) + d_{\bs{e}}(b'_F, b_F) \\
&\leq n + n + m_0 + n + n = 4n + m_0.
\end{align*}
\end{proof}

%[Later state conseq. of the theorem,, e.g. $T$ is G-cpt over $\bs e$ iff  there is $n$ such that for any $F$, 
%$d_{\bs e}(x_F,y_F)\leq n$ ...]

%With Proposition \ref{propcharofgcpt}, we summarize equivalent conditions for $T$ being G-compact over $\bs{e}$.
%In particular, using Theorem \ref{typedeflstp}, we conclude that in fact there's a uniform $n < \omega$ that bounds diameter of any hyperimaginary if $T$ is G-compact over $\bs{e}$.

In addition to Proposition \ref{propcharofgcpt}, we now supply more conditions
equivalent to $T$ being G-compact over $\bs e$. In particular, 
by applying  Theorem \ref{typedeflstp} we extend Newelski's result on the uniform finite bound for the diameters of real tuples,  mentioned in the beginning of this section, to the hyperimaginary context. 
% Namely, $T$ is G-compact over $\bs e$ iff there exists a uniform finite bound for the diameters of any hyperimaginaries over $\bs e$.

%In particular using Theorem \ref{typedeflstp}, we conclude that there exists a uniform $n < \omega$ which bounds the diameters of any hyperimaginaries over $\bs e$ if $T$ is G-compact over $\bs{e}$.

\begin{Corollary}\label{corgcptchar}
The following are equivalent.
\begin{enumerate}
    \item $T$ is G-compact over $\bs{e}$.
   % \item For any hyperimaginaries $b_F$ and $c_F$, $b_F \equiv^{\KP}_{\bs{e}} c_F$ iff $b_F \equiv^{\LL}_{\bs{e}} c_F$.
    %\item For any small real tuples $b$ and $c$, $b \equiv^{\KP}_{\bs{e}} c$ iff $b \equiv^{\LL}_{\bs{e}} c$.
    %\item For any $\emptyset$-type-definable equivalence relation $F$, $x_F \equiv^{\LL}_{\bs{e}} y_F$ is type-definable.
    \item For any hyperimaginary $b_F$, there is a partial type $\Psi(x)$ (over $ab$)
     such that $c\models \Psi(x)$ iff $c_F\equiv^{\LL}_{\bs e}b_F$. 
  
    \item For any hyperimaginary $b_F$, there is $n < \omega$ such that for any $c_F \equiv^{\LL}_{\bs{e}} b_F$, we have $d_{\bs{e}}(b_F, c_F) \leq  n$.

    \item There is a uniform $n < \omega$ such that for every pair of  hyperimaginaries $b_F$ and $c_F$, $b_F \equiv^{\LL}_{\bs e} c_F$ iff $d_{\bs{e}}(b_F, c_F) \leq n$.
    %\item There is a uniform $n < \omega$ such that for any (bounded length of) sequence of hyperimaginaries $\bs{b} = (b_i/F_i \mid i \in I)$ and $\bs{c} = (c_i/F_i \mid i \in I)$, $\bs{b} \equiv^{\LL}_{\bs{e}} \bs{c}$ iff $d_{\bs{e}}(\bs{b}, \bs{c}) \leq n$.
\end{enumerate}
\end{Corollary}

\begin{proof}
%The equivalence of $(1)$ through $(4)$ is due to Proposition \ref{propcharofgcpt}.
%\newline
(1) $\Rightarrow $ (2), (4) $\Rightarrow $ (1):
By Proposition \ref{propcharofgcpt}(4).
\medskip

(2) $\Rightarrow$ (3): By  Theorem \ref{typedeflstp}.
\medskip 

(3) $\Rightarrow$ (4):
Assume (3).
To lead a contradiction, suppose not (4),  that is for each $n$ there is a pair of hyperimaginaries 
$b_n/F_n$, $c_n/F_n$ such that $b_n/F_n\equiv^{\LL}_{\bs e} c_n/F_n$ but  $d_{\bs e}(b_n/F_n, c_n/F_n)> n$ \ (*).  
%\newline
%\indent 

Without loss of generality we can  assume $b_n\equiv^{\LL}_{\bs e} c_n$.
Consider $b=(b_n\mid n<\omega)$, and let $F=\bigwedge_n F_n$.
As pointed out in the first paragraph of Section 1, $b_F$ is a single hyperimaginary.
Then by (3), there must be $m$ such that for any real $d$ with $|d| = |b|$, we have $b_F\equiv^{\LL}_{\bs e} d_F$ iff $d_{\bs e}(b_F,d_F)\leq m$. Now since $b_m \equiv^{\LL}_{\bs e} c_m$, there is 
$f\in \autfe(\CM)$ such that $f(b_m)=c_m$.
We now let $c_F=f(b_F)$ where $c=(c'_n\mid n<\omega)$ and $c'_n=f(b_n)$ (so $c_m=c'_m$).
Thus $b_F\equiv^{\LL}_{\bs e} c_F$, and $d_{\bs e}(b_F,c_F)\leq m$.
In particular $d_{\bs e}(b_m/F_m,c_m/F_m)\leq m$, contradicting (*).
%\newline
%$(6) \Rightarrow (7)$:
%Again, as in $(5) \Rightarrow (6)$, let $b_F$ and $c_F$ be hyperimaginaries given by the equivalence classes of $F(x,y) \equiv \bigwedge_{i \in I}F(x_i,y_i)$.
%Then by $(6)$, $d_{\bs{e}}(b_F, c_F) \leq n$ so that there are $b_F = b_0/F, \cdots, b_n/F = c_F$ and $\bs{e}$-indiscernible sequences $I_l$ for $l < n$ where $b_l/F, b_{l+1}/F \in I_l$.
%But then replacing each $b_l/F$ with a sequence of hyperimaginaries $\bs{b}_l$, $I_l$ becomes an $\bs{e}$-indiscernible sequence of sequences of hyperimaginaries containing $\bs{b}_l$ and $\bs{b}_{l+1}$.
%Thus $d_{\bs{e}}(\bs{b}, \bs{c}) \leq n$.
%\newline
%$(7) \Rightarrow (3)$:
%Define $\Gamma_n(x_0,x_n) \leftrightarrow \exists x_1 \cdots x_{n-1}(\bigwedge_{i < n}\Gamma(x_i,x_{i+1}))$ where $\Gamma$ is defined at Definition \ref{Gamdef}.
%Then for any real tuples $b$ and $c$, $\Gamma_n(b,c)$ iff $d_{\bs{e}}(b,c) \leq n$ iff $b \equiv^{\LL}_{\bs{e}} c$, i.e. $\equiv^{\LL}_{\bs{e}}$ is type-definable for real tuples of any arity.
%Now $\equiv^{\LL}_{\bs{e}}$ is an $\bs{e}$-invariant type-definable bounded equivalence relation (Remark \ref{preli}), hence $b \equiv^{\KP}_{\bs{e}} c$ iff $b \equiv^{\LL}_{\bs{e}} c$ by Proposition \ref{propcharofKP}(2).
\end{proof}

% Summing up, we get :

% \begin{Corollary}
% The following are equivalent.
% \begin{enumerate}
%     \item $T$ is G-compact over $\bs{e}$.
%     \item For any hyperimaginary $c_F$, there is a partial type $\Psi(x)$ (over $ac$) such that $d \models \Psi(x)$ iff $d_F \equiv_{\bs e}^{\LL} c_F$.
%     \item There is a uniform $n < \omega$ such that for every pair of  hyperimaginaries $c_F$ and $d_F$, $c_F \equiv^{\LL}_{\bs e} d_F$ iff $d_{\bs{e}}(c_F, d_F) \leq n$.
% \end{enumerate}
% \end{Corollary}

% $(1) \Rightarrow (2)$, $(3) \Rightarrow (1)$ : These follow from Proposition \ref{propcharofgcpt}(1)$\Leftrightarrow$(4).

% $(2) \Rightarrow (3)$ Assume (2).  To lead a contradiction suppose not (3).  Hence for each $n$ there is a pair of  hyperimaginaries 
% $c_n/F_n$, $d_n/F_n$ such that  $c_n/F_n\equiv^{\LL}_{\bs e} d_n/F_n$ but  $d_{\bs e}(c_n/F_n, d_n/F_n)> n$ \ (*).  
% In addition we can clearly assume $c_n\equiv^{\LL}_{\bs e} d_n$.
% We now consider
% $c=(c_n\mid n<\omega)$, and let $F=\bigwedge_n F_n$. As pointed out before Definition \ref{strongaut},
% $c_F$ is a  hyperimaginary. Then  by (2) and Theorem \ref{typedeflstp}, there must be $m$ such that 
% $c_F\equiv^{\LL}_{\bs e} b_F$ iff $d_{\bs e}(c_F,b_F)\leq m$. Now since $c_m\equiv^{\LL}_{\bs e} d_m$, there is 
% $f\in \autfe(\CM)$ such that $f(c_m)=d_m$. We now
%  let $d_F=f(c_F)$  where $d=(d'_n\mid n<\omega)$ and $d'_n=f(c_n)$ (so $d_m=d'_m$).  Thus  $c_F\equiv^{\LL}_{\bs e} d_F$,
%  and $d_{\bs e}(c_F,d_F)\leq m$. In particular, $d_{\bs e}(c_m/F_m,d_m/F_m)\leq m$ contradicting (*).

\section{Relativized Lascar groups}

In \cite{DKL}, the notions of relativized Lascar groups for real types are introduced.
In this last section, we generalize the definitions for the hyperimaginary  types, and  
supply a partial positive answer (Theorem \ref{thmkernelfinite}) to a question raised in \cite{DKKL}, which is even a new result in the real context. We will use Proposition \ref{propcharofgcpt} 
%(or Corollary \ref{corgcptchar})
 in proving the result. 
%Using the result of previous section  , 
%In this section 
%We start with the definition of relativized Lascar groups over a hyperimaginary.
Throughout this section we fix a hyperimaginary $b_F$ and $p = \tp(b_F/\bs{e})$.

\begin{Definition}    $ $
%For any hyperimaginary $b_F$ and $p = \tp(b_F/\bs{e})$,
\begin{enumerate}
    \item $\aut(p) =\aute(p):= \{f \upharpoonright p(\CM) \mid f \in \aute(\CM)\}$.
    \item For a cardinal $\lambda > 0$, $\autf^{\lambda}(p) =\autfe^{\lambda}(p):=$
    \[\{f \in \aut(p) \mid \text{ for any }\overline{b_F} = (b_i/F)_{i < \lambda}\text{ such that }b_i/F \models p,\text{ }\overline{b_F} \equiv^{\LL}_{\bs{e}} f(\overline{b_F})\}.\]
    \item $\autf^{\operatorname{fix}}(p) = \autfe^{\operatorname{fix}}(p):=$
    \[\{f \in \aut(p) \mid \text{ for any }\lambda \text{ and }\overline{b_F} = (b_i/F)_{i < \lambda}\text{ such that }b_i/F \models p,\text{ }\overline{b_F} \equiv^{\LL}_{\bs{e}} f(\overline{b_F})\}.\]
    \item $\gall^{\lambda}(p) = \aut(p) / \autf^{\lambda}(p)$.
    \item $\gall^{\operatorname{fix}}(p) = \aut(p) / \autf^{\operatorname{fix}}(p)$.
\end{enumerate}
\end{Definition}

\begin{remark}$ $
\begin{enumerate}
\item Obviously,  if $f \in \autfe(\CM)$ then $f\restriction p(\CM) \in \autf^{\lambda}(p)$.
    \item One can easily check that $\autf^{\lambda}(p), \autf^{\operatorname{fix}}(p)$ are normal subgroups of $\aut(p)$.
    \item There is an example in \cite{DKKL} that $\gall^1(p)$ and $\gall^2(p)$ are distinct for $p\in S(\emptyset)$.
\end{enumerate}
\end{remark}

%\begin{Remark}\label{rellasomega}
%$\autf^{\operatorname{fix}}(p) = \autf^{\omega}(p)$, so $\gall^{\operatorname{fix}}(p) =\gall^{\omega}(p) $.
%\end{Remark}

We point out the following basic facts  which can be proved by the same arguments in \cite{DKL}. (A detailed proof may be found in \cite{Le}.)

%By the following Remark and Proposition, $\gall^{\lambda}(p)$ is a topological group which is independent of the choice of a monster model, as in \cite{DKL}:

\begin{Fact}\label{rellastopgp}$ $
\be
\item
$\autf^{\operatorname{fix}}(p) = \autf^{\omega}(p)$, so $\gall^{\operatorname{fix}}(p) =\gall^{\omega}(p) $.

\item
Up to isomorphism, the group structure of $\gall^{\lambda}(p)$ is independent of the choice of a monster model.
\item
Recall that  $\bs{e} \in \dcl(M)$. Then $\nu' : S_M(M) \rightarrow \gall^{\lambda}(p)$ defined by $\nu'(\tp(f(M)/M)) = (f \upharpoonright p(\CM)) \cdot \autf^{\lambda}(p)$ is well-defined and $\gall^{\lambda}(p)$ is a quasi-compact topological group with the quotient topology given by $\nu'$, which is independent of the choice of $M$.
\ee
\end{Fact}

%\begin{proof}
%$\nu'$ is well-defined by Remark \ref{preli}(3) a fortiori.
%Notice that the topology must coincide with the quotient topology induced by the projection map $\xi : \gall(T,\bs{e}) \rightarrow \gall^{\lambda}(p,\bs{e})$ and the topology of $\gall(T,\bs{e})$ is independent of the choice of $M$.
%Then since a quotient of a topological group by a normal subgroup is again a topological group, the result follows.
%\end{proof}

%\begin{prop}\label{rellasindmon}
%Up to isomorphism, the group structure of $\gall^{\lambda}(p,\bs{e})$ is independent of the choice of a monster model.
%\end{prop}

We use a boldface letter $\bs b$ to denote the fixed hyperimaginary $b_F$, so the type $p(x)=\tp(\bs b/\bs e)$. 
Now 
%fix $p(x) = \tp(\bs b/\bs{e})$ where $\bs b=b/F$ is a hypermiaginary, and 
let $\bs c=c_L$ be a  hyperimaginary such that $\bs c \in \acl(\bs b \bs{e})$.
Say $\bs c_0(=c_0/L), \cdots, \bs c_{m-1}= c_{m-1}/L$ ($1\leq m$) are all the distinct $\bs b\bs{e}$-conjugates of $\bs c$.
{\bf We assume that $\bs b \bs c_i \not\equiv^{\LL}_{\bs{e}} \bs b\bs c_j$ for any distinct $i,j < m \quad (\dag)$} until Theorem \ref{thmkernelfinite}.
Put
\[\overline{p}(xy) = \tp(\bs b\bs c/\bs{e}) \equiv \exists z_1z_2w( \tp_{z_1z_2w}(bca) \wedge F(x, z_1) \wedge L(y, z_2) \wedge E(a, w)).\]
In the rest, $\bs b\bs c$ with attached small scripts refers to a  realization of $\p(x,y)$. Moreover for example, given  $\bs b'\bs c'$, the corresponding plain letter 
$b'c'$ refers to a real tuple such that $\bs b'=b'/F$ and $\bs c'=c'/L$, so that $b'c'$ is a real realization of   $\p(x,y)$ and vice versa. 
Notice that $\bs b'\bs c'\equiv_{\bs e} \bs b\bs c$, while  $b'c'\equiv_{\bs e}bc$ {\em need not} hold.

For each  $\bs b' \models p$, we fix $\bs c^{\bs b'}_0 \cdots \bs c^{\bs b'}_{m-1}$ which is   the image of $\bs c_0 \cdots \bs c_{m-1}$  under some $\bs e$-automorphism sending $\bs b$ to $\bs b'$. 
% the all conjugates of $b_Fc_L$ by `some' $\bs e$-automorphism mapping $b_F$ to $b'_F$.
We may just write $\bs b'\bs c^{\bs b'}$ to refer to some  $\bs b'\bs c^{\bs b'}_{j}$ ($j<m$).
As usual we write $x_{< n}$ to denote the sequence of variables $x_0 \cdots x_{n-1}$ and 
%extensively use similar abbreviations for other objects to reduce heavy strings of symbols.
we use similar abbreviations for  finite sequences of hyperimaginaries. 
%We point out that 
%$(\bs b\bs c)_{<n}= (\bs b'\bs c')_{<n}$ {\em need not} imply $(b c)_{<n}\equiv_{\bs e} ( b' c')_{<n}$.

\begin{nota} $ $
\begin{enumerate}
\item For $T$ G-compact over $\bs{e}$, 
%and any tuples of variables $x,y$ of the same arity, 
we write 
$E^{\LL}(-,-,a)$ to denote an $\bs{e}$-invariant type over $a$ defining 
the Lascar equivalence over $\bs e$ of hyperimaginaries (see Proposition \ref{propcharofgcpt}). We may assume that $E^{\LL}$ is closed under finite conjunctions and  every formula $\phi(x,y,a) \in E^{\LL}$ is reflexive and symmetric.
%$x_F \equiv^{\LL}_{\bs{e}} y_F$.
    %\item For $T$ G-compact over $\bs{e}$ and any tuples of variables $x,y$ of the same arity, $E^{\LL}(x_F,y_F,a)$ denotes an $\bs{e}$-invariant type over $a$ defining $x_F \equiv^{\LL}_{\bs{e}} y_F$.
    \item $\pi_n : \gall^n(\overline{p}) \rightarrow \gall^n(p)$ is the natural projection. Namely $\pi_n((f\restriction \ov p(\CM))\cdot \autf^n(\ov p))=(f\restriction p(\CM))\cdot \autf^n(p)$, which clearly  is well-defined. 
    \item $K_n$ is the kernel of $\pi_n : \gall^n(\overline{p}) \rightarrow \gall^n(p)$.  We write $\ov f$ to denote $f\cdot \autf^n(\ov p)\in \gall^n(\ov p).$
\end{enumerate}
\end{nota}

In \cite{DKKL}, it is asked whether $K_n$ is finite for finite $n>0$ (when both $\bs b, \bs e$ are real) if $T$ is G-compact over $\bs e$. 
In this section we positively answer the question under the assumption ($\dag$) (even for hyperimaginaries).

%\begin{remark}\label{remerfsym}
%We may assume that $E^{\LL}$ is closed under finite conjunctions and  every formula $\phi(x,y,a) \in E^{\LL}$ is reflexive and symmetric,
%( and notice that $bd \models \phi$ iff $b'd' \models \phi$ for $b'_Fd'_F = b_Fd_F$.
%\end{remark}

\begin{Lemma}\label{lemlstpfml}
Assume that $T$ is G-compact over $\bs{e}$ and that  $(\dag)$ holds.
Then given positive $n < \omega$, there is a formula
\[\alpha_n((xy)_{< n}, (x'y')_{< n}, a) \in E^{\LL}((x_Fy_L)_{< n}, (x'_Fy'_L)_{< n}, a)\]
such that if
\[\models (\bigwedge_{i < n}\overline{p}(\bs b'_i,\bs c'_i) \wedge \overline{p}(\bs b_i'',\bs c_i'')) \wedge \bs b'_{< n}\equiv^{\LL}_{\bs e} \bs b''_{< n} \wedge \alpha_n((b'c')_{< n}, (b''c'')_{< n}, a),\]
then $\models E^{\LL}((\bs b'\bs c')_{< n}, (\bs b''\bs c'')_{< n}, a)$.
\end{Lemma}

\begin{proof}
To lead a contradiction, suppose that there is no such $\alpha_n$.
Then for each $\alpha \in E^{\LL}((x_Fy_L)_{< n}, (x'_Fy'_L)_{< n}, a)$, there are $\bs b'_i\bs c^{\bs b'_i},\  \bs b''_i\bs c^{\bs b''_i} \models \overline{p}$ ($i < n$) such that
\begin{align*}
\models  \bs b'_{< n}\equiv^{\LL}_{\bs e} \bs b''_{< n} \wedge \alpha((b'c^{b'})_{< n}, (b''c^{b''})_{< n}, a) \text{ but } \\
\neg E^{\LL}((\bs b'\bs c^{\bs b'})_{< n}, (\bs b''\bs c^{\bs b''})_{< n}, a)
\end{align*}
where for example, 
$$(b'c^{b'})_{< n}=(b'_ic^{b'_i}\mid i<n) \text{ and }
(\bs b'\bs c^{\bs b'})_{< n}=(\bs b'_i\bs c^{\bs b'_i}\mid i<n).$$

% the $n$-tuple of realizations of $\overline{p}$, $(b'_ic^{b'_i}_{j_i} \mid i < n)$ is denoted as $(b'c^{b'})_{< n}$ (this notation is okay because we are not interested in the particular index of $c$) and similarly for $(b'_Fc^{b'_F}_L)_{< n}$.

Now  since $\bs b'_{< n}\equiv^{\LL}_{\bs e} \bs b''_{< n} $, there are conjugates $\bs b'_i\bs c_*^{\bs b'_i}$ $(i<n)$ satisfying
\[\models E^{\LL}((\bs b'\bs c^{\bs b'}_*)_{< n}, (\bs b''\bs c^{\bs b''})_{< n}, a),\]
 so that  $(\bs c^{\bs b'_i}_*\mid i<n) \ne (\bs c^{\bs b'_i}\mid i< n).$
Note now that since $\neg L(c_j,c_\ell)$ for each pair $j < \ell (< m)$, there is $\psi_{j\ell}(y,z) \in L(y,z) \subseteq \CL$ such that $\neg \psi_{j\ell}(c_j,c_\ell)$.
Put $\psi(y,z) \equiv \bigvee_{j < \ell < m}\neg \psi_{j\ell}(y,z)$.
Then we must have  $(c^{b'_i}, c^{b'_i}_*\mid i<n) \models \Psi (y_i,z_i\mid i<n)$, where
$$ \Psi (y_i,z_i\mid i<n):= \exists (y'z')_{<n}( \bigwedge_{i<n}L(y_i, y'_i)\wedge \bigwedge_{i<n} L(z_i,z'_i)\wedge\bigvee_{i < n}\psi(y'_i,z'_i)).$$

In conclusion, the following type $\bigwedge_{i < n}(\overline{p}(x_i,y_i) \wedge \overline{p}(x_i',y_i'))\wedge$
%\begin{align*}
$$E^{\LL}((x_Fy_L)_{< n}, (x'_Fy'_L)_{< n}, a) \wedge E^{\LL}((x_Fz_L)_{< n}, (x'_Fy'_L)_{< n}, a) \wedge \Psi((yz)_{<n})$$
%\wedge \bigwedge_{i < n}(\overline{p}(x_i,y_i) \wedge \overline{p}(x_i',y_i'))
%\end{align*}
of variables $(xy)_{<n}, (x'y'z)_{<n}$
is finitely satisfiable,  witnessed by $(b' c^{ b'})_{< n}, ( b'' c^{ b''}c_*^{b'})_{< n}.$
Hence by compactness, there are   realizations $(bc)_{<n}$ and $(bc')_{<n}$ of  $(x_Fy_L)_{<n}$  and $(x_Fz_L)_{<n}$ of the type, respectively. 
Thus we have  $(\bs b\bs c)_{<n}\equiv^L_{\bs e} (\bs b \bs c')_{<n}$, which  implies $\bs b_i\bs c_i\equiv^{\LL}_{\bs e} \bs b_i\bs c'_i$ for all $i<n$. 
But since $\models \Psi((cc')_{<n})$, there must be some
$i_0<n$ such that $\bs c_{i_0}\ne \bs c'_{i_0}$, contradicting 
 our assumption $(\dag)$.
\end{proof}

% \begin{remark}\label{remalphaeinv}
% Since $E^{\LL}((x_Fy_L)_{< n}, (x'_Fy'_L)_{< n}, a)$ is $\bs{e}$-invariant, it follows that in Lemma \ref{lemlstpfml}, for any $a' \models \tp(a/\bs{e})$, $\alpha_n((xy)_{< n}, (x'y')_{< n}, a')$ works the same as $\alpha_n((xy)_{< n}, (x'y')_{< n}, a)$.
% \end{remark}

\begin{Theorem}\label{thmkernelfinite}
If $T$ is G-compact over $\bs{e}$ and $(\dag)$ holds, then $K_n$ is finite for each  finite $n \geq 1$.
\end{Theorem}

\begin{proof}
We start with a claim.  We work with the formula $\alpha_n$  obtained in Lemma \ref{lemlstpfml}. 

\begin{Claim1}
There is $\phi_n((xy)_{< n}, (x'y')_{< n}, a) \in E^{\LL}((x_Fy_L)_{< n}, (x'_Fy'_L)_{< n}, a)$ such that for any $a_0, a_1,a_2 \models q(z):=\tp(a/\bs{e})$,
\begin{align*}
\phi_n((xy)_{< n}, (x'y')_{< n}, a_0) \wedge  \phi_n((x'y')_{< n}, (uv)_{< n}, a_1)\wedge \phi_n((uv)_{< n}, (x''y'')_{< n}, a_2) \\ \models \alpha_n((xy)_{< n}, (x''y'')_{< n}, a).
\end{align*}
\end{Claim1}

\begin{proof}[Proof of Claim 1]
Suppose not the claim, that is, for each $\phi((xy)_{< n}, (x'y')_{< n}, a) \in E^{\LL}((xy)_{< n}, (x'y')_{< n}, a)$, there are $a_0, a_1,a_2 \models q$ such that
\begin{align*}
\phi((xy)_{< n}, (x'y')_{< n}, a_0) \wedge \phi((x'y')_{< n}, (uv)_{< n}, a_1)\wedge \phi((uv)_{< n}, (x''y'')_{< n}, a_2) \\ \wedge \neg \alpha_n((xy)_{< n}, (x''y'')_{< n}, a)
\end{align*}
is consistent.
Then by compactness,
\begin{align*}
q(z_0) \wedge q(z_1)\wedge q(z_2)\wedge E^{\LL}((xy)_{< n}, (x'y')_{< n}, z_0) \wedge E^{\LL}((x'y')_{< n}, (uv)_{< n}, z_1) \\
\wedge E^{\LL}((uv)_{< n}, (x''y'')_{< n}, z_2) \wedge \neg \alpha_n((xy)_{< n}, (x''y'')_{< n}, a)
\end{align*}
is consistent.
But $E^{\LL}((xy)_{< n}, (x'y')_{< n}, a)$ is $\bs{e}$-invariant, thus in fact we get
\[(xy)_{< n} \equiv^{\LL}_{\bs{e}} (x''y'')_{< n} \wedge \neg \alpha_n((xy)_{< n}, (x''y'')_{< n}, a),\]
which is impossible since $\alpha_n((xy)_{< n}, (x''y'')_{< n}, a) \in E^{\LL}((xy)_{< n}, (x''y'')_{< n}, a)$.
\end{proof}

Let $ \{(\bs b^\ell\bs c^\ell)_{<n} \mid \ell \in I\}$ 
%(or for simplicity, $\{(b_F)^l_{< n} \mid l \in I\}$) 
be a (small) set of all the  representatives of  $\equiv^{\LL}_{\bs{e}}$-classes of realizations of $\bigwedge_{i<n}\p(x_i,y_i)$.
Thus for any $(b'c')_{< n} \models \bigwedge_{i<n}\overline{p}(x_i,y_i)$, there is 
$\ell \in I$ such that 
%$(\bs b^l\bs c^{\bs b^l})_{<n}=(\bs b^l_0\bs c^{\bs b^l_0},\dots, \bs b^l_{n-1}\bs c^{\bs b^l_{n-1}}) \models \bigwedge_{i<n}\overline{p}(x_i,y_i)$
%$$(\bs b^l\bs c^{\bs b^l})_{<n}=(\bs b^l_0\bs c^{\bs b^l_0},\dots, \bs b^l_{n-1}\bs c^{\bs b^l_{n-1}}) \models \bigwedge_{i<n}\overline{p}(x_i,y_i)$$
  $\models E^{\LL}((b'c')_{< n}, (b^\ell c^\ell)_{< n}, a)$, so $\models \phi_n((b'c')_{< n}, (b^\ell c^{\ell})_{< n}, a)$ a fortiori.
Hence by compactness 
%(of the space of types $S(aB)$ and its closed subset), 
%we can say 
there are  $(\bs b^0\bs c^0)_{<n},\dots,(\bs b^{k-1}\bs c^{k-1})_{<n}$ such that
\[\bigwedge_{i<n}\overline{p}(x_i,y_i) \vdash \bigvee_{\ell <k} \phi_n((xy)_{<n},  (b^\ell_ic^{\ell}_{i}\mid i<n),a). \quad (*)\]
%(Recall here $\bs c_0^{\bs b^l_i},\dots,\bs c_{m-1}^{\bs b^l_i} $ are $\bs b^l_i\bs e$-conjugates of $\bs c_0^{\bs b^l_i}$.)
%In the rest, we write $\ov f$ to denote $f\cdot \autf^n_{\bs e}(\ov p)\in \gall^n(\ov p,\bs e).$
(Here some of $\bs b^\ell$'s $(\ell< k)$ could be the same.)
Recall that for $i<n$, $\bs c_0^{\bs b^\ell_i},\dots,\bs c_{m-1}^{\bs b^\ell_i} $ are the fixed $\bs b^\ell_i\bs e$-conjugates of $\bs c_0^{\bs b^\ell_i}$. Hence
indeed 
%there is $\iota_\ell\in {}^nm$ such that 
$\bs c^\ell_i=\bs c^{\bs b^\ell_i}_{i_0}$ for some $i_0<m$.

Now let $\ov f\in K_n$.  Thus $f(\bs b^\ell_{<n})\equiv^{\LL}_{\bs e}\bs b^\ell_{<n}$. Then  for each $\ell <k$, there is 
 $\iota_\ell \in {}^nm$ such that  $f((\bs b^\ell\bs c^\ell)_{<n})\equiv^{\LL}_{\bs e}(\bs b^\ell_i \bs c^{\bs b^\ell_i}_{\iota_{\ell} (i)}\mid i<n) .$
 Moreover if $\ov f=\ov g$ then  $f((\bs b^\ell\bs c^\ell)_{<n})\equiv^{\LL}_{\bs e}g((\bs b^\ell\bs c^\ell)_{<n})$. Hence due to our assumption $(\dag)$,  
  the mappings $\ell(<k)\mapsto \iota_\ell$ for $f$ and $g$ must be the same.  Therefore by the following claim we can   conclude that 
  $|K_n|$ is finite  $\leq (m^n)^k=m^{nk}$, since there are only $m^n$-many choices of $\iota_\ell$ for each $\ell <k$.

\begin{Claim2}
Let $\ov f,\ov g\in K_n$.  Suppose that the described mappings $\ell\mapsto \iota_\ell$ (for all $\ell <k$) above   for $\ov f$ and $\ov g$ are the same. Then $\ov f=\ov g$.
\end{Claim2}

\begin{proof}[Proof of Claim 2]
Let $(\bs b'\bs c')_{<n}\models \bigwedge_{i<n} \bar p(x_iy_i)$. By $(*)$, there is some  $l <k$ and $(b^lc^l)_{<n}$ satisfying 
$\phi_{n}((b'c')_{<n}, (b^l c^l)_{<n},a)$ \  $(**)$.  
%Hence  we have 
  %$$\phi_{n}(f((b'c')_{<n}), f(b^l_ic^{b^l_i}_{j_i}\mid i<n),f(a))\mbox{ and }\phi_{n}(g((b'c')_{<n}), g(b^l_ic^{b^l_i}_{j_i}\mid i<n),g(a)).$$
Thus 
$$\phi_n(f((b'c')_{<n}), f((b^l c^l)_{<n}),f(a))\mbox{ and } \phi_n(g((b'c')_{<n}), g((b^lc^l)_{<n}),g(a)).$$

Now due to the supposition in Claim 2, we have 
 $$f((\bs b^l\bs c^l)_{<n})\equiv^{\LL}_{\bs e}(\bs b^l_i \bs c^{\bs b^l_i}_{\iota_l (i)}\mid i<n)\equiv^{\LL}_{\bs e}g((\bs b^l\bs c^l)_{<n}).$$
Hence there is $h\in \autfe(\CM)$ such that  $f((\bs b^l \bs c^l)_{<n})=hg((\bs b^l\bs c^l)_{<n})$.
%??? there is no such uniform $h$ for all $\ell <k$??????
Then $f((b^lc^l)_{<n})=hg((b^l c^l)_{<n})$ may not hold, while 
$\phi_{n}( f((b^lc^l)_{<n}), hg((b^l c^l)_{<n}, a)$ must hold. Moreover by $(**)$, we have 
$\phi_n(hg((b'c')_{<n}), hg((b^lc^l)_{<n}),hg(a)).$

Thus by Claim 1 (with $\phi_n$ being symmetric), it follows that 
$$\models \alpha_n(f(b'c')_{<n}, hg((b'c')_{<n}),a).$$
Now since $\ov f, \ov g\in K_n$, we have $f(\bs b'_{<n})\equiv^{\LL}_{\bs e}hg(\bs b'_{<n})\equiv_{\bs e}^{\LL} \bs b'_{<n}$.
Hence Lemma \ref{lemlstpfml} implies that $f((\bs b'\bs c')_{<n})\equiv^{\LL}_{\bs e}hg((\bs b'\bs c')_{<n})\equiv_{\bs e}^{\LL}g((\bs b'\bs c')_{<n})$.
Since  $(\bs b'\bs c')_{<n}$ is an arbitrary realization of $\bigwedge_{i<n} \bar p(x_iy_i)$, we conclude that $\ov f=\ov g$.
\end{proof}
\end{proof}

In the remaining section we confirm that two real case results on $K_1$ in 
\cite{DKKL} can be extended, by  following  essentially the same proofs,  to the hyperimaginary context (Corollary \ref{k1finite}, Proposition \ref{propabelian}).  

%, it is also shown  that $K_1$ is finite  when $\bs b, \bs e$ are real,

%Theorem \ref{thmkernelfinite} above  is proved  without the assumption of $(\dag)$ when $n=1$ and $\bs b, \bs e$ are real. 
%We confirm  that  the same holds even in the hyperimaginary context (Corollary \ref{k1finite}).  We essentially follow the same proofs. 
%If $n = 1$,

%In \cite{DKKL}, Theorem \ref{thmkernelfinite} above  is proved  without the assumption of $(\dag)$ when $n=1$ and $\bs b, \bs e$ are real. 
%We confirm  that  the same holds even in the hyperimaginary context (Corollary \ref{k1finite}).  We essentially follow the same proofs. 
%If $n = 1$,
\medskip
 
We now remove the assumption $(\dag)$, and 
recall that $\bs c=\bs c_0, \cdots, \bs c_{m-1}$ are all the distinct 
$\bs b\bs{e}$-conjugates of $\bs c$.
Possibly reordering them,  we assume that {\bf $\bs b\bs c_0, \cdots, \bs b\bs c_{m_0-1}$ $(1\leq m_0\leq m)$ are  the representatives of all the distinct $\equiv^{\LL}_{\bs{e}}$-classes in $\overline{p} = \tp(\bs b\bs c/\bs{e})$} with the first coordinate $\bs b$.

\begin{Lemma}\label{lemlstpfml1}
If $T$ is G-compact over $\bs{e}$, then there is a formula
\[\alpha(xy, x'y', a) \in E^{\LL}(x_Fy_L, x'_Fy'_L, a)\]
such that if
\[\models \overline{p}(\bs b',\bs c') \wedge \overline{p}(\bs b'',\bs c'') \wedge \bs b'\equiv^{\LL}_{\bs e} \bs b'' \wedge \alpha(b'c', b''c'', a),\]
then $\models E^{\LL}(\bs b'\bs c', \bs b''\bs c'', a)$.
\end{Lemma}

\begin{proof} If $m_0=1$ then any formula in $E^{\LL}$ would work. Hence we assume $m_0>1$. Now   
the type 
\begin{align*}
E^{\LL}(x_Fy_L, x'_Fy'_L, a) \wedge \exists zwz'w'(z_Fw_Lz'_Fw'_L \equiv_{\bs{e}} x_Fy_Lx'_Fy'_L \\ \wedge \bigvee_{0\leq i < j < m_0}E^{\LL}(z_Fw_L,\bs b\bs c_i,a) \wedge E^{\LL}(z'_Fw'_L,\bs b\bs c_j,a))
\end{align*}
is inconsistent. Thus there is $\alpha(xy,x'y',a) \in E^{\LL}(x_Fy_L,x'_Fy'_L,a)$ such that 
\begin{align*}
\alpha(xy,x'y',a) \wedge \exists zwz'w'(z_Fw_Lz'_Fw'_L \equiv_{\bs{e}} x_Fy_Lx'_Fy'_L \\ \wedge \bigvee_{0\leq i <j < m_0}E^{\LL}(z_Fw_L,\bs b\bs c_i,a) \wedge E^{\LL}(z'_Fw'_L,\bs b\bs c_j,a))
\end{align*}
is inconsistent.
Assume $b'c'b''c'' \models \overline{p}(xy) \wedge \overline{p}(x'y') \wedge \alpha(xy,x'y',a) \wedge x_F\equiv^{\LL}_{\bs e} x'_F$.
Notice that there is $c^*b^*c^{**}$ such that $\bs b'\bs c'\bs b''\bs c'' \equiv_{\bs{e}} \bs b\bs c^*\bs b^*\bs c^{**}$. Thus $\bs b^*\equiv^{\LL}_{\bs{e}} \bs b$.
Now $\bs b\bs c^* \equiv^{\LL}_{\bs{e}} \bs b\bs c_i$ and $\bs b^*\bs c^{**} \equiv^{\LL}_{\bs{e}} \bs b\bs c_j$ for some $i,j < m_0$.
But by the choice of $\alpha$, it must be $i = j$ and hence  $\bs b'\bs c' \equiv^{\LL}_{\bs{e}}\bs b''\bs c''$.
\end{proof}

 The following
can be  obtained
by  exactly the same proof of Theorem \ref{thmkernelfinite} with $n = 1$ (in this case ($\dag$) is needless) and $\alpha$ in Lemma \ref{lemlstpfml1}, which we will not repeat. 
%We omit the detail. 

\begin{Corollary}\label{k1finite}
If  $T$ is G-compact over $\bs{e}$, then $K_1$ is finite.
\end{Corollary}

%\begin{proof}
%Apply exactly the same proof of Theorem \ref{thmkernelfinite} with $n = 1$ and $\alpha$ in Lemma \ref{lemlstpfml1}. 
%\end{proof}

\begin{remark}\label{remabelian}
Assume that $\gall^1(p)$ is abelian. Let $f \in \aute(p)$ and let $b_0 \models p$.
If $f(\bs b_0) \equiv^{\LL}_{\bs{e}} \bs b_0$, then for any $b' \models p$, $f(\bs b') \equiv^{\LL}_{\bs{e}} \bs b'$.
\end{remark}

\begin{proof}
There is $g \in \aute(p)$ such that $g(\bs b_0) = \bs b'$.
Now since $\gall^1(p)$ is abelian, $f(\bs b') = f(g(\bs b_0)) \equiv^{\LL}_{\bs{e}} g(f(\bs b_0)) \equiv^{\LL}_{\bs{e}} g(\bs b_0) = \bs b'$.
\end{proof}

\begin{Proposition}\label{propabelian}
If $\gall^1(\overline{p})$ is abelian, then $|K_1| = m_0$.
\end{Proposition}

\begin{proof}
Let $\ov f, \ov g \in \gall^1(\overline{p})$ and assume $f(\bs b\bs c) \equiv^{\LL}_{\bs{e}} g(\bs b\bs c)$.
Then there is $h \in \autfe(\CM)$ such that $hf(\bs b\bs c) = g(\bs b\bs c)$, thus $g^{-1}hf(\bs b\bs c) = \bs b\bs c$.
By Remark $\ref{remabelian}$ with $\overline{p}$, we have $g^{-1}hf \in \autfe^1(\overline{p})$. Hence $\ov f  = \ov g $. In addition  if $\ov f, \ov g \in K_1$ then 
$f(\bs b)\equiv^{\LL}_{\bs e} g(\bs b)\equiv^{\LL}_{\bs e}\bs b$.  Therefore 
we get $|K_1| = m_0$,
since there are only $m_0$-many distinct $\equiv^{\LL}_{\bs{e}}$-classes in $\overline{p}$ with the first coordinate Lascar equivalent to $\bs b$ over $\bs{e}$.
\end{proof}

Now several   observations  in \cite{DKKL} automatically follow in our context. For example, if  $T$ is G-compact over $\bs e$ then $\pi_n$ is a covering homomorphism
 when $(\dag)$ holds, or $n=1$. We point out  Theorem \ref{last} as well, which relies on 
 a purely compact group theoretical result from \cite{DKKL}.

\begin{Fact}
Let $G$ be an abelian compact connected topological group, and let $F$ be a finite subgroup of $G$. Then $G$ and $G/F$ are isomorphic as topological groups.
\end{Fact}

\begin{Theorem}\label{last}
Assume $T$ is G-compact over $\bs e$; and  $\acl(\bs e)$ and $\bs e$ are interdefinable. If 
$n=1$ or $(\dag)$ holds; and   
$\gall^n(\ov p)$ is abelian, then it is isomorphic to $\gall^n(p)$ as topological groups. 
\end{Theorem}

\end{document}